\documentclass[letterpaper, 16pt]{article}

\newcommand{\dateofsubmission}{03.09.12}

\usepackage{graphicx, color}
\usepackage{amsfonts}
\usepackage{amsmath}
\usepackage{amssymb}
\usepackage{amsthm}
\usepackage{subfigure}
\usepackage{pslatex}
\usepackage{enumitem}
\usepackage{stmaryrd}
\usepackage{epstopdf}

\newtheorem{theorem}{Theorem}[section]
\newtheorem{definition}[theorem]{Definition}

\newtheorem{lemma}[theorem]{Lemma}

\newtheorem{question}[theorem]{Question}
\newtheorem{example}[theorem]{Example}

\DeclareMathOperator{\str}{str}
\DeclareMathOperator{\sstr}{sstr}
\DeclareMathOperator{\vc}{vc}
\DeclareMathOperator{\dvc}{dvc}
\DeclareMathOperator{\set}{SET}
\DeclareMathOperator{\fshift}{fullshift}
\DeclareMathOperator{\dshift}{dn}
\DeclareMathOperator{\sign}{sign}
\DeclareMathOperator{\spn}{span}

\newcommand{\ignore}[1]{}

\newcommand{\powset}{\mathcal{P}}
\newcommand{\setdef}[2]{\{#1~~\vert~~#2\}}

\renewcommand{\restriction}{\mathord{\upharpoonright}}

\newcommand{\s}{\mathbb{S}}
\newcommand{\sysnot}[1]{\mathbb{#1}}
\newcommand{\system}[2]{\langle #1,\{0,1\}^{#2}\rangle }
\newcommand{\sys}[2]{\langle #1,#2\rangle }
\newcommand{\restr}[2]{#1\restriction_{#2}}

\newcommand{\normalize}[1]{n(#1)}
\newcommand{\normalizealone}{n}
\newcommand{\syscup}{~\uplus~}
\newcommand{\syscap}{~\nplus~}
\newcommand{\bigsyscup}{~\biguplus~}
\newcommand{\bigsyscap}{~\bignplus~}
\newcommand{\inter}[1]{\bignplus_{#1}}
\newcommand{\union}[1]{\biguplus_{#1}}
\newcommand{\mymerge}{\star}
\newcommand{\defeq}{\stackrel{def}{=}}
\newcommand{\deftext}[1]{{\bf {#1}}}
\newcommand{\smaller}{\preceq}

\makeindex

\begin{document}

\pagestyle{empty}



\vfill\vfill\vfill
\begin{center}
	\LARGE\textbf{Saarland University}
\end{center}
\begin{center}
	\LARGE\textbf{Faculty of Natural Sciences and Technology I}
\end{center}
\begin{center}
	\LARGE\textbf{Department of Computer Science}
\end{center}
\vfill\vfill\vfill
\begin{center}
	\Large Master thesis
\end{center}
\vfill
\begin{center}
	\Huge\textbf{Shattering Extremal Systems}
\end{center}
\vfill\vfill\vfill
\begin{center}
	submitted by\\[1mm]
	\large Shay Moran
\end{center}
\vfill
\begin{center}
	submitted\\
	\large\dateofsubmission
\end{center}
\vfill\vfill\vfill
\begin{center}
	Supervisor\\[1mm]
	\large Prof. Dr. Ami Litman
\end{center}
\vfill
\begin{center}
	Supervisor\\[1mm]
	\large Prof. Dr. Kurt Mehlhorn
\end{center}
\vfill
\begin{center}
	Reviewers\\[1mm]
	\large Prof. Dr. Kurt Mehlhorn\\
	\large Prof. Dr. Ami Litman
\end{center}
\strut
\clearpage















\pagestyle{headings}
\setcounter{page}{1}
\pagenumbering{roman}
\tableofcontents

\cleardoublepage

\setcounter{page}{1}
\pagenumbering{arabic}


\section{Introduction}
The {\it Shatters} relation and the {\it VC dimension} have been investigated since the early seventies (starting
with \cite{Sauer,Shelah,VC1}). These concepts have found numerous applications in
statistics, combinatorics, learning theory and computational geometry.

{\it Shattering extremal systems} are set-systems with a very rich structure and many different characterizations. The goal of this paper is to elaborate on the structure of these systems. They were discovered several times and independently by several groups of researchers. Lawrence in \cite{Law} is the first who introduced them in his study of convex sets. Intrestingly, the definition he gave does not require the concept of shatters. Independently, Bollob{\'a}s et al in \cite{BR95} have discovered these systems, using the {\it shatters} relation (a.k.a {\it traces}); furthermore, Bollob{\'a}s et al are the first who introduced the relation of {\it strongly-shatters} (a.k.a {\it strongly-traces}) and characterized Shattering extremal systems with the shatters and strongly-shatters relations. Dress et al, independently of Bollobas et al, have discovered the same characterization and established the equivalence to the characterization given by Lawrence.

It is an interesting phenomena that, in the past thirty years, different characterizations of these systems were discovered in different fields of pure and applied mathematics. Strangely, there seem to be no connections between the different groups that studied these systems. We hope to link together these different groups by this work and as a result to enhance the research of these systems. Here is a list of some of the contexts in which these systems were discovered: Functional analysis \cite{Pajor}, Discrete-geometry (\cite{Law}), Phylogenetic Combinatorics (\cite{Dress1,Dress2}) and Extremal Combinatorics (\cite{BR89,BR95}). Moreover, as this work shows, the class of Shattering extremal systems naturally extends the class of {\it Maximum systems}\footnote{set-systems that meet Sauer's inequality with equality}. Maximum systems occur in Geometry \cite{Welzl} and in Learning theory \cite{WF95,WK07,RR1,RR2,RR3}. Thus, it is not unlikely to find usage of Shattering extremal systems in these fields. 

\subsection{Our results}
We present new definitions of {\it shatters} and {\it strongly-shatters} that differ from each other only in the order of the quantifiers. We demonstrate that at least some of the known duality between these two concepts is due to this transpose of quantifiers. We shed additional light on this mysterious phenomena of duality via a mechanical transformation that, sometimes, translates claims to dual claims and proofs to dual proofs.

Two unary operators, which we denote $\inter{}$ and $\union{}$, were extensively used under many different notations in the literature. These operators preserve the property of shattering extremality and the property of being a maximum system. Hence, these operarators were studied in both contexts. Many previous works observed a certain duality between these two operators. We demonstrate that this duality stems from the famous duality of Boolean algebra. Moreover, we reveal that a previously known lemma concerning duality is just the De-Morgan laws.

We present the famous {\it down shifting} operators in the terminology of oblivious sorting algorithms. The most common use of these operators is to apply a sequence of them on a system until a fixed point is reached. Our presentation of down-shifting sheds light on this process and relates it to a basic problem in the theory of sorting.

We introduce new characterizations of the {\it shatters} and {\it strongly shatters} relations. One characterization is by the operators $\inter{}$ and $\union{}$ and the other characterization is by down-shifting operators. As a corollary we obtain simple proofs to two known characterizations of Shattering extremal systems via commutativity between $\inter{}$ and $\union{}$ and commutativity between down-shifting operators. Moreover, we introduce a weaker form of Shattering-Extremality and establish equivalences of this form with a weaker form of the above commutativities.

We show that, similarly to Maximum systems, Shattering extremal systems occur naturally in Geometry as a combinatorial feature of an arrangement of oriented hyper-planes. We give a geometrical interpretation of {\it shattering} and {\it strong-shattering} in this setup. We also introduce the class of {\it ``Convex systems''} which form a simple and natural generalization of the classes studied by Welzl (\cite{Welzl}) and of the classes studied by Lawrence (\cite{Law}). We pose as an open question whether every Shattering Extremal system is a convex system..

Finally, we use this theory to develop a machinery for proving certain equalities and inequalities concerning undirected graphs and their orientations. As an example of this machinery we will prove the following inequality:\\
Let $G=(V,E)$ be an undirected graph and let $s,t$ be two distinct vertices in $V$. Then, the number of orientations of $E$ that contain a directed path from $s$ to $t$ is the same as the number of subgraphs\footnote{i.e. subsets of $E$} of $G$ that contain an undirected $s-t$ path.\\
One advantage of this machinery is that it demonstrates an application of {\it Strong shattering} which is exceptional since most applications of this theory use {\it shattering}.



\newpage
\section{Shattering Extremal systems - A first date}\label{sec:intro}
\subsection{Basic definitions}
This work studies finite set-systems and thus all structures discussed in this paper are finite.
A standard definition of a set-system is a pair $(T,X)$ where $X$ is a set and $T \subseteq \powset(X)$. For reasons that will become clear later, we choose a different definition for set-systems. Also, we prefer to call these structures ``systems'' rather than ``set-systems''.
\begin{definition}[System]\label{def:set-sys} 
A \deftext{system} is a pair, $\s=\system{S}{X}$, where $X$ is a set and $S\subseteq \{0,1\}^X$. For a system $\s=\system{S}{X}$:
\begin{itemize}
	\item{Let $S(\s)$, $\dim(\s)$ and $C(\s)$ denote $S$, $X$ and $\{0,1\}^X$ respectively.}
	\item{Let $|\s|$ denote $|S|$.}
	\item{The \deftext{complement} of $\s$, denoted $\neg\s$, is the system defined by:\begin{center} $\neg\s\defeq\sys{C(\s)-S(\s)}{C(\s)}.$\end{center}}
	\item{The elements of $C(\s)$ are called {\it vertices}}
\end{itemize}
\end{definition}
Note that for every set $X$ there are exactly $2^{2^{\lvert X\rvert}}$ systems, $\s$, with $C(\s)=\{0,1\}^X$.\\
In particular, there are exactly two systems for $X=\emptyset$, one of them is empty and the other contains one vertex.
Let $\mathcal{K}_0$ denote the former and let $\mathcal{K}_1$ denote the latter system.
\begin{definition}
Let $f,g$ be arbitrary functions and let $A\subseteq Dom(f)\cap Dom(g)$.
	\begin{itemize}
	\item{\deftext{$f$ agrees with $g$ on $A$} means that $ f(x)=g(x)$ for all $x\in A$.}
	\item{\deftext{$f$ agrees with $g$} means that $f$ agrees with $g$  on $Dom(f)\cap Dom(g)$}
	\end{itemize} 
\end{definition}

\begin{definition}[Cube]
Let $X$ be a set and let $Y\subseteq X$. A \deftext{$Y$-cube} of $\{0,1\}^X$ is an equivalence class of the following equivalence relation on $\{0,1\}^X$: ``$u$ agrees with $v$ on $X-Y$''
	\begin{itemize}
	\item{Let $C$ be a $Y$-cube of $\{0,1\}^X$. The \deftext{dimensions of} $C$ are defined to be the set $Y$ and denoted by $\dim(C)$.}
	\end{itemize}
\end{definition}
Note that there are $2^{|X-Y|}$ $Y$-cubes, they are disjoint and cover all of $\{0,1\}^X$.

\subsection{Lopsided systems}
Lawrence introduced the concept of lopsided systems as follows
\begin{definition}[Lawrence \cite{Law}]\label{def:lopsided}
A system $\s$ is called \deftext{lopsided} if for every $\{X',X''\}$, a partitioning of $\dim(\s)$, exactly one of the following statements is true:
\begin{enumerate}
\item{$S(\s)$ contains an $X'$-cube} 
\item{$S(\neg\s)$ contains a $X''$-cube}
\end{enumerate}
\end{definition}
The following important property of lopsided set-sytems is an immediate corollary of the definition:
\begin{theorem}[Complementing]\label{thm:lopsided and complementing}
A system, $\s$, is lopsided if and only if $\neg\s$ is lopsided.
\end{theorem}

\subsection{Shatters and Strongly-Shatters}
\begin{definition}[Merging]
Let $f\in \{0,1\}^X$ and $g\in \{0,1\}^Y$ where $X,Y$ are disjoint.
We define $f\mymerge g$  to be the unique function in $\{0,1\}^{X\cup Y}$ that agrees with both $f$ and $g$.
\end{definition}
Note that $\mymerge$ is a commutative and associative operation.

Let $\s$ be a system, let $X=\dim(\s)$ and let $Y \subseteq \dim(\s)$.
\begin{definition}[Shatters]
We say that $\s$ \deftext{shatters} Y if $$({\forall f\in\{0,1\}^Y})({\exists g\in\{0,1\}^{X-Y}}): g\mymerge f\in S(\s)$$
\end{definition}

\begin{definition}[Strong-shatters]
We say that $\s$ \deftext{strongly shatters} $Y$ if $$({\exists g\in\{0,1\}^{X-Y}})({\forall f\in\{0,1\}^Y}):g\mymerge f\in S$$
\end{definition}

That is, our definitions of {\it strongly shatters} and {\it shatters} are identical except of the order of the two quantifiers. These definitions highlight the duality that was observed by previous works.
For example, a straightforward application of predicate calculus on the above definitions gives the next lemma: 
\begin{lemma}[\cite{BR95,ShatNews}]\label{obs:dualStrSstr}
Let $\s$ be a system and let $\{X',X''\}$, be a partitioning of $\dim(\s)$. Then exactly one of the following statements is true:
	\begin{enumerate}
	\item{$\s$ shatters $X'$}
	\item{$\neg\s$ strongly shatters $X''$}
	\end{enumerate}
\end{lemma} 
Note the similarity between Lemma \ref{obs:dualStrSstr} and definition \ref{def:lopsided}. That is, the former is derived from the latter by replacing ``{\it shatters}'' in Statement $1$ with ``{\it strongly-shatters}''.

A natural variant of ``{\it Lopsided}'' is the following concept:
\begin{definition}[Dual lopsided system]
A system $\s$ is called \deftext{Dual lopsided} if for every $\{X',X''\}$, a partitioning of $\dim(\s)$, exactly one of the following statements is true:
	\begin{enumerate}
	\item{$\s$ shatters $X'$} 
	\item{$\neg\s$ shatters $X''$}
	\end{enumerate}
\end{definition}

Two subsets of $\powset(\dim(\s))$ are associated with $\s$: 
\begin{itemize}
	\item{$\str(\s)\defeq\setdef{Y\subseteq \dim(\s)}{Y \mbox{ is shattered by }\s}$}
	\item{$\sstr(\s)\defeq\setdef{Y\subseteq \dim(\s)}{Y \mbox{ is strongly shattered by }\s}$}
\end{itemize}
Obviously, both $\str(\s)$ and $\sstr(\s)$ are closed under subset relation and $\sstr(\s)\subseteq \str(\s)$. This work studies systems for which $\sstr(\s)=\str(\s)$.

\subsection{Shattering extremal systems}
\begin{definition}[Shattering-Extremal system]
$\s$ is \deftext{Shattering-Extremal}\footnote{This is not the standard definition of Shattering Extremal systems. The standard definition and its equivalence to this definition will be discussed in Section \ref{sec:SE systems and Sandwich theorem}} (in abbreviation: SE) if it satisfies $$\sstr(\s)=\str(\s).$$
The property ``$\s$ is SE'' is sometimes refered to by the expression ``$SE(\s)$''.
\end{definition}

\subsection{Shattering extremality = lopsidedness}
This section proves the following theorem:
\begin{theorem}\label{thm:SE=Lopsided}
The following statements are equivalent for a system $\s$:
	\begin{enumerate}
	\item{$SE(\s)$}
	\item{$SE(\neg\s)$}
	\item{$\s$ is lopsided}
	\item{$\s$ is dual lopsided}
	\end{enumerate}
\end{theorem}
The equivalence of the first three statements was proven by Dress et al in \cite{Dress1}.
\begin{lemma}\label{lem:lopsided iff SE}
Let $\s$ be a system. $\s$ is lopsided $\iff$ $\sstr(\s)=\str(\s)$
\end{lemma}
\begin{proof}
	\begin{align*}
	Lopsided(\s)	&\iff \forall X\in \powset(\dim(\s)): \dim(\s)-X\notin \sstr(\neg\s)\leftrightarrow X\in \sstr(\s) &\mbox{by definition}\\
								&\iff \forall X\in \powset(\dim(\s)): X\in \str(\s)\leftrightarrow X\in \sstr(\s) &\mbox{by Lemma \ref{obs:dualStrSstr}}\\
								&\iff \sstr(\s)=\str(\s)
	\,.
	\end{align*}
\end{proof}
As noted earlier, there is a certain duality in this theory between shattering and strong shattering. This duality is manifested by a certain mechanical tranformation on text written in ``mathematical english''. The texts of interest are lemmas and proofs and we refer to the transformed texts as dual lemmas and dual proofs. This dual transformation swaps few words and symbols as follows: It swaps the pair ``$\str$'' and ``$\sstr$'', the pair ``Lopsided'' and ``Dual lopsided'', the pair ``$\subseteq$'' and ``$\supseteq$'' and the pair ``$\leq$'' and ``$\geq$''. This dual transformation is useful since, sometimes, the dual of a true lemma is true. Moreover, sometimes the dual of a proof is a valid proof of the dual lemma. Lemma \ref{lem:lopsided iff SE} and its proof are an example to this phenomenon.
\begin{lemma}[Dual of Lemma \ref{lem:lopsided iff SE}]\label{lem:Dual lopsided iff SE}
Let $\s$ be a system. Then \begin{center} $\s$ is dual lopsided $\iff$ $\sstr(\s)=\str(\s)$. \end{center}
\end{lemma}
It is easy to verify that the dual of the proof of Lemma \ref{lem:lopsided iff SE} is a valid proof of Lemma \ref{lem:Dual lopsided iff SE}.

These two lemmas establish equivalence between items $1,3,4$ in Theorem \ref{thm:SE=Lopsided}. The rest of the theorem is derived from Theorem \ref{thm:lopsided and complementing}.

\newpage
\section{Inequalities}\label{sec:inequalities}
This chapter presents basic inequalities which involve the concepts of {\it Shattering} and {\it Strong-Shattering} and discusses their relation with $SE$ systems.\\
\subsection{VC-dimension and Sauer's lemma}
\begin{definition}[VC-dimension]
The \deftext{VC-dimension} (Vapnik Chervonenkis dimension) of a system $\s$, denoted $\vc(\s)$, is the cardinality of the largest\footnote{As a special case, $\vc(\s)=-1$ when $S(\s)=\emptyset$} subset that is shattered by it. 
\end{definition}
The well-known ``Sauer lemma'' was proved in the 70's by Sauer (\cite{Sauer}), Shelah (\cite{Shelah}) and 
Vapnik and Chervonenkis (\cite{VC1}):
\begin{theorem}[Sauer's lemma]\label{thm:Sauer lemma}
Let $\s$ be a system with $|\dim(\s)|=n$. Then $$|\s|\leq \sum_{i=0}^{\vc(\s)}{{n}\choose{i}}$$
\end{theorem}
A proof for this lemma will be given in the next section.

\subsection{The Sandwich theorem}
The following generalization of Sauer's lemma is the result of an accumulated work by several authors. Different parts of this theorem were proven independently several times. (Pajor \cite{Pajor}, Bollob{\'a}s and Radcliffe \cite{BR95}, Dress \cite{Dress1} and Holzman and Aharoni \cite{ShatNews,Greco98})
\begin{theorem}[Sandwich theorem \cite{Pajor,BR95,Dress1,ShatNews,Greco98}]\label{thm:basic inequality}
For a system $\s$:
$$\lvert \sstr(\s)\rvert \leq \lvert\s\rvert \leq \lvert \str(\s)\rvert$$
\end{theorem}
This section discusses and proves this theorem. The following trivial fact will be useful. For any $a',b',a'',b''\in\mathbb{R}$:
	\begin{equation}\label{eq:real numbers fact}
	a'\leq b'\mbox{ and }~a''\leq b''\mbox{ and }a'+b'=a''+b''~~\implies~~ a'=b'\mbox{ and }a''=b'' 
	\end{equation}
\newpage
\begin{lemma}\label{lem:central lemma str}
Let $\s,\s',\s''$ be systems that satisfy the following:
	\begin{enumerate}
	\item{$\lvert\s'\rvert+\lvert\s''\rvert=\lvert\s\rvert$}
	\item{$\lvert \str(\s')\rvert+\lvert \str(\s'')\rvert\leq\lvert \str(\s)\rvert$}
	\item{$\lvert \s'\rvert\leq\lvert \str(\s')\rvert,~\lvert \s''\rvert\leq\lvert \str(\s'')\rvert $}
	\end{enumerate}
Then:
	\begin{enumerate}[label=(\alph*)]
	\item{$\lvert\s\rvert\leq\lvert \str(\s)\rvert$}
	\item{$\lvert\s\rvert=\lvert \str(\s)\rvert$ $\iff$ Assumptions $2$ and $3$ in the premise of the lemma are satisfied with equality}
	\end{enumerate}
\end{lemma}
Remark: Assumption 3 is always true as will be proved later
\begin{proof}
Regarding Conclusion $(a)$:
	\begin{align*}
	\lvert \str(\s) \rvert	&\geq\lvert \str(\s')\rvert + \lvert \str(\s'')\rvert	&\mbox{by Assumption $2$}\\
	&\geq\lvert\s'\rvert + \lvert\s''\rvert																		&\mbox{by Assumption $3$}\\
	&=\lvert\s\rvert																													&\mbox{by Assumption $1$}\\
	\,.
	\end{align*}
Considering Conclusion $(b)$: The right to left direction is trivial. Assume $\lvert\s\rvert=\lvert \str(\s)\rvert$. Therefore, the above chain of inequalities holds  when the symbol `$\geq$' is replaced with `$=$'. This implies equality in Assumption $2$. Fact (\ref{eq:real numbers fact}) implies equality in Assumption $3$.
\end{proof}
\begin{definition}\label{def:restrictions on x}
Let $\s$ be a system, let $x\in \dim(\s)$. The \deftext{restrictions of $\s$ associated with $x$} are the following two systems:
	\begin{itemize} 
	\item{$\sys{\{\restr{f}{\dim(\s)-\{x\}}~:~f\in\s,~f(x)=0\}}{\{0,1\}^{\dim(\s)-\{x\}}}$}
	\item{$\sys{\{\restr{f}{\dim(\s)-\{x\}}~:~f\in\s,~f(x)=1\}}{\{0,1\}^{\dim(\s)-\{x\}}}$}
	\end{itemize}
\end{definition}
Let $\s$ be a system, let $x\in \dim(\s)$ and let $\{\s',\s''\}$ be the pair of restrictions of $\s$ associated with $x$. We will prove that $\lvert\s\rvert\leq\lvert \str(\s)\rvert$ by showing that these particular $\s,\s',\s''$ satisfy the premise of Lemma \ref{lem:central lemma str}.

\begin{definition}
Let $A,B$ be sets and $x$ be an object. Let $\oplus_{x}$ denote the following operator:
	$$
	A\oplus_{x} B\defeq A\cup B\cup\big\{ c\cup\{x\}~:~c\in A\cap B\big\}
	$$
\end{definition}
The following two lemmas are straightforward.
\begin{lemma}\label{lem:oplus trivial}
Let $A,B$, be sets and let $x$ be an object such that $x\notin T$ for every $T\in A\cup B$. Then
	$$
	\lvert A\oplus_{x}B\rvert=\lvert A\rvert+\lvert B\rvert
	$$
\end{lemma}
\begin{lemma}\label{lem:str and oplus}
Let $\s$ be a system, let $x\in \dim(\s)$ and let $\{\s',\s''\}$ be the pair of restrictions of $\s$ associated with $x$. Then:
	$$
	\str(\s')\oplus_{x}\str(\s'')\subseteq \str(\s)
	$$
\end{lemma}
\begin{lemma}\label{lem:upper inequality}
Let $\s$ be a system. Then: $$|\s| \leq |\str(\s)|.$$
\end{lemma}
\begin{proof}
We prove the claim by induction on $\dim(\s)$. The case of $\dim(\s)=\emptyset$ is trivial. Otherwise, pick $x\in \dim(\s)$ and let $\{\s',\s''\}$ be the pair of restrictions of $\s$ associated with $x$. It is enough to show that $\s',\s''$ satisfy the three Assumptions of Lemma \ref{lem:central lemma str}. Assumption $1$ is immediate. Assumption $2$ follows from Lemmas \ref{lem:oplus trivial} and \ref{lem:str and oplus} Assumption $3$ follows by the induction hypothesis.
\end{proof}

To see that Lemma \ref{lem:upper inequality} generalize Sauer lemma, note that by the definition of $\vc$ dimension:
$$\str(\s)\subseteq\{Y\subseteq \dim(\s)~:~\lvert Y\rvert\leq \vc(\s) \}.$$
Hence, Sauer lemma (Theorem \ref{thm:Sauer lemma}) is an easy conclusion of Lemma \ref{lem:upper inequality}.

To establish the Sandwich theorem it remains to show that $\lvert \sstr(\s) \rvert\leq\lvert \s \rvert$ and this is accomplished via applying the Duality tranformation on the proofs of Lemmas \ref{lem:central lemma str} and \ref{lem:str and oplus} and \ref{lem:upper inequality}. It is easy to verify that the duals of our proofs for these lemmas are valid\footnote{Note that the Dual transformation of a proof transforms only the text explicitly written. It does not translate implicit arguments}. The dual of Lemma \ref{lem:upper inequality} gives the desired inequality, namely, 
$$\lvert \sstr(\s) \rvert\leq\lvert \s \rvert.$$
This finishes the proof of Theorem \ref{thm:basic inequality}.

Our proof for the Sandwich Theorem demonstrates the usefulness of the Duality Transformation. However, it is important to write some lines about the fragile nature of this tranformation. This transformation works only for some lemmas/theorems. It is not hard to find (true) claims whose dual claims are not true. Moreover, as we will see in the next section (Lemma \ref{lem:upper implies SE}), there are some cases in which the dual of a lemma is true but the dual of the presented proof is not valid. Thus, when using this transformation, one has to take special care and to check the validity of the arguments being transformed.

\subsection{SE systems and the Sandwich theorem}\label{sec:SE systems and Sandwich theorem}
The inequalities in the sandwich theorem suggest two more extremal properties of systems, namely:
\begin{itemize}
\item{Systems $\s$ for which $\lvert \sstr(\s)\rvert=\lvert\s\rvert$}
\item{Systems $\s$ for which $\lvert\s)\rvert=\lvert \str(\s)\rvert$}
\end{itemize}

The following result concerning the equivalence of these extremal properties was proven independently in \cite{Dress2}, and in \cite{BR95}:
\begin{theorem}\label{thm:SE char}
Let $\s$ be a system. The following statements are equivalent:
	\begin{enumerate}
	\item{$\sstr(\s)=\str(\s)$}
	\item{$\lvert \sstr(\s)\rvert=\lvert \s\rvert$}
	\item{$\lvert \s\rvert=\lvert \str(\s)\rvert$}
	\end{enumerate}
\end{theorem}
This section discusses and proves Theorem \ref{thm:SE char}.
\begin{lemma}\label{lem:upper implies SE}
For every system $\s$: $\lvert \str(\s)\rvert=\lvert \s\rvert\implies \sstr(\s)=\str(\s)$
\end{lemma}
\begin{proof}
The lemma is proven by induction on $\lvert \dim(\s)\rvert$. Let $B\in \str(\s)$. It is enough to show that $B\in \sstr(\s)$. The case of $B=\dim(\s)$ is easy. Otherwise, pick $x\notin B$ and let $\s',\s''$ be the two restrictions of $\s$ associated with $x$. By Lemma \ref{lem:central lemma str}, $\s',\s''$ satisfy the premise of the current lemma. By the induction hypothesis:
\begin{equation}\label{eq:SE one side}
\sstr(\s')=\str(\s')
\end{equation}
\begin{equation}\label{eq:SE other side}
\sstr(\s'')=\str(\s'')
\end{equation}
Moreover, Lemma \ref{lem:central lemma str} also implies that $\lvert \str(\s')\rvert+\vert \str(\s'')\rvert=\lvert \str(\s)\rvert$. This, combined with Lemmas \ref{lem:oplus trivial} and \ref{lem:str and oplus}, implies that
\begin{equation}\label{eq: oplus of restrictions is str}
\str(\s')\oplus_{x}\str(\s'')= \str(\s)
\end{equation}
Therefore:
	\begin{align*}
  B\in \str(\s) &\implies B\in \str(\s')\oplus_x \str(\s'')   &\mbox{by Equation \ref{eq: oplus of restrictions is str}}\\
							 &\implies B\in \str(\s')\cup \str(\s'')			 &\mbox{since }x\notin B\\
							 &\implies B\in \sstr(\s')\cup \sstr(\s'')     &\mbox{by Equations (\ref{eq:SE one side}) and (\ref{eq:SE other side}) }\\
							 &\implies B\in \sstr(\s) 
							 \,.
	\end{align*}
\end{proof}
The dual of Lemma \ref{lem:upper implies SE} is true. However, we dont know a proof for Lemma \ref{lem:upper implies SE} whose dual is valid\footnote{The dual of our proof of Lemma \ref{lem:upper implies SE} is not valid as observed, for example, by the dual of the second line in the proof: ``Let $B\in \sstr(\s)$. It is enough to show that $B\in \str(\s)$... ''. This argument is obviously wrong.}. Therefore, we present a different proof for the dual lemma.

Lemma \ref{obs:dualStrSstr} implies that for all $\s$: $\lvert \str(\s)\rvert+\lvert \sstr(\neg\s)\rvert=2^{\lvert \dim(\s)\rvert}$. Clearly, $\lvert\s\rvert+\lvert\neg\s\rvert=2^{\lvert \dim(\s)\rvert}$. This implies:
\begin{lemma}\label{lem:uep(s) is lep(sc)}
For every system $\s$, the following statements are equivalent:
	\begin{enumerate}
	\item{$\lvert\s\rvert=\lvert \str(\s)\rvert$}
	\item{$\lvert \sstr(\neg\s)\rvert=\lvert\neg\s\rvert$}
	\end{enumerate}
\end{lemma}

\begin{lemma}[Dual of Lemma \ref{lem:upper implies SE}]\label{lem:lower implies SE}
For every system $\s$: 
$$\lvert \sstr(\s)\rvert=\lvert \s\rvert\implies \sstr(\s)=\str(\s).$$
\end{lemma}
\begin{proof}
	\begin{align*}
	\lvert \sstr(\s)\rvert=\lvert\s\rvert &\implies \lvert \str(\neg\s)\rvert=\lvert\neg\s\rvert &\mbox{by Lemma \ref{lem:uep(s) is lep(sc)}}\\
																		   &\implies \sstr(\neg\s)=\str(\neg\s) &\mbox{by Lemma \ref{lem:upper implies SE}}\\
																			 &\implies \sstr(\s)=\str(\s) &\mbox{by Theorem \ref{thm:SE=Lopsided}}\\
																			 \,.
	\end{align*}
\end{proof}
The Sandwich Theorem and Lemmas \ref{lem:upper implies SE}, \ref{lem:lower implies SE} finish the proof of Theorem \ref{thm:SE char}.

By Theorem \ref{thm:SE char} it follows that any of the 3 properties mentioned there, may serve as a definition of {\it Shattering Extremality}.\\
The most common definition in the literature is $\lvert\s\rvert=\lvert \str(\s)\rvert$. ( See \cite{BR95, ShatNews, Greco98})\\
We prefer $\sstr(\s)=\str(\s)$ as this definition is invariant under duality. 

\subsection{A recursive characterization of $SE$}
\begin{theorem}\label{thm:RRR}
Let $\s$ be a system, let $x\in \dim(\s)$ and let $\{\s',\s''\}$ be the pair of restrictions of $\s$ associated with $x$. Then the following statements are equivalent:
	\begin{enumerate}
	\item{$SE(\s)$.}
	\item{$SE(\s')$ and $SE(\s'')$ and $\str(\s')\oplus_x \str(\s'')=\str(\s)$.}
	\item{$SE(\s')$ and $SE(\s'')$ and $\sstr(\s')\oplus_x \sstr(\s'')=\sstr(\s)$.}
	\end{enumerate}
\end{theorem}
\begin{proof}
The proof is divided into two parts.

First part: $1\iff 2$. By Lemmas \ref{lem:oplus trivial} and \ref{lem:str and oplus} it follows that 
$$\str(\s')\oplus_x \str(\s'')=\str(\s) \iff \lvert \str(\s')\rvert+\lvert \str(\s'')\rvert=\lvert \str(\s)\rvert.$$
This, combined with Conclusion $(b)$ of Lemma \ref{lem:central lemma str} and Theorem \ref{thm:SE char}, concludes the first part of this proof.

Second part: $1\iff 3$. This part is the dual of the first part. It is easy to verify that the dual of the proof given for the first part is valid.
\end{proof}

\subsection{Maximum systems}
The class of Maximum systems is the ``famous relative'' of the class of SE-systems.
\begin{definition}[Maximum systems]
$\s$ is called a \deftext{Maximum system} if it meets Sauer's lemma with equality: $$|\s|= \sum_{i=0}^{\vc(\s)}{{n}\choose{i}}$$
\end{definition}
Maximum systems and their applications in machine learning have been discussed in \cite{RR1,RR2,RR3,WK07,WF95}. The relation between maximum systems and arrangements of hyper-planes in a Euclidean space have been discussed in \cite{Welzl}

SE systems are an extension of maximum systems (i.e every maximum system is SE). Moreover, SE systems seem to possess all ``nice properties'' of maximum systems and some other ``nice properties'' that maximum systems do not have. (See for example Theorem \ref{thm:preservation under restrictions}) \\

\newpage
\section{Operations on systems}

\subsection{Pre-systems and restrictions}
Let $\s$ be a system and let $C\subseteq C(\s)$ be a cube. Consider the structure $\sys{S(\s)\cap C}{C}$. According to our terminology this structure is not a system. We refer to it as a ``pre-system''. Namely:
\begin{definition}[pre-system]
Let $X$ be a set. A pre-system is a pair $\sys{S}{C'}$ where $C'$ is a cube of $\{0,1\}^X$ and $S\subseteq C'$
\end{definition}
Clearly, every pre-system $\sys{S}{C'}$ can be translated to a system by restricting its vertices to $\dim(C')$ as follows:
\begin{definition}[Normalizing]
Let $X$ be a set and let $C\subseteq\{0,1\}^X$ be a cube. Define $\normalizealone:C\rightarrow\{0,1\}^{\dim(C)}$ by
	$$
	\normalize{f}=\restr{f}{\dim(C)}
	$$
This mapping is called {\it the normalizing of $C$}.
\end{definition}
this mapping is naturally extended to subsets of $C$ as follows, Let $S\subseteq C$. Then $\normalize{S} \defeq\{\normalize{f}~\vert~f\in S\}$; and to pre-systems. Let $\s=\sys{S}{C}$ be a pre-system. Then	$\normalize{\s} \defeq\sys{\normalize{S}}{\{0,1\}^{\dim(C)}}$. Note that the normalization of a pre-system is a system.
\begin{definition}[restriction]
Let $\s$ be a system, let $C'\subseteq C(\s)$ be a cube. Let $\restr{\s}{C'}$ denote the normalization of the pre-system $\sys{S(\s)\cap C'}{C'}$. This system is called the \deftext{restriction} of $\s$ to $C'$
\end{definition}
\subsubsection{Restrictions preserve SE}
The following important property of $SE$ systems is an immediate corollary of Lemma \ref{thm:RRR}
\begin{theorem}\label{thm:preservation under restrictions}
Let $\s$ be a system, and let $C$ be a cube. Then: $$SE(\s)\implies SE(\restr{\s}{C})$$ 
\end{theorem}
The property of ``being a maximum system'' is not preserved by restrictions. Therefore this theorem is an example of the advantage of SE systems over maximum systems.

\subsection{Boolean operations} 
\begin{definition}[Boolean operations]
Let $\sysnot{A},\sysnot{B}$ be two systems with $\dim(\sysnot{A})=\dim(\sysnot{B})=X$. The \deftext{union} and \deftext{intersection} of $\sysnot{A},\sysnot{B}$ is defined as follows:
\begin{enumerate}
	\item{$\sysnot{A}\syscup \sysnot{B}\defeq\sys{S(\sysnot{A})\cup S(\sysnot{B})}{\{0,1\}^X}$}
	\item{$\sysnot{A}\syscap \sysnot{B}\defeq\sys{S(\sysnot{A})\cap S(\sysnot{B})}{\{0,1\}^X}$}
\end{enumerate}
\end{definition}
Note that if $\dim(\sysnot{A})\neq \dim(\sysnot{B})$ then these operations are undefined and that both $\syscup,\syscap$ are associative and commutative binary operations.
These two operations and the complement operator (\ref{def:set-sys}) satisfy the relations of the (corresponding) operations in a boolean algebra.\\ 
As an example, the following is a trivial variant of the famous De-Morgan's laws:
\begin{enumerate} 
\item{$\neg(\sysnot{A}\syscup\sysnot{B})=(\neg\sysnot{A})\syscap(\neg\sysnot{B})$}
\item{$\neg(\sysnot{A}\syscap\sysnot{B})=(\neg \sysnot{A})\syscup(\neg\sysnot{B})$}
\end{enumerate}

These operations are used as unary operators in the standard way:
Let $\mathcal{T}$ be a collection of systems such that $\forall\s,\s'\in\mathcal{T}:\dim(\s)=\dim(\s')$. $\bigsyscup{\mathcal{T}}~$ denotes the union over all elements of $\mathcal{T}$ and $\bigsyscap{\mathcal{T}}~$ denotes the intersection over all elements of $\mathcal{T}$.

\subsection{The Boolean operators $\inter{Y}$ and $\union{Y}$}
\begin{definition}[$\inter{Y}$ and $\union{Y}$]\label{def:bool ops}
	Let $Y$ be a set. The following operators, $\inter{Y},\union{Y}$ are defined only on systems $\s$ such that $Y\subseteq \dim(\s)$. In this case:
 	\begin{enumerate}
 		\item{$$\inter{Y}(\s)\defeq\bigsyscap{\{\restr{S}{C}~\vert~C\mbox{ is a cube with }\dim(C)=\dim(\s)-Y\}}$$}
 		\item{$$\union{Y}(\s)\defeq\bigsyscup{\{\restr{S}{C}~\vert~C\mbox{ is a cube with }\dim(C)=\dim(\s)-Y\}}$$}
 	\end{enumerate}
 	For an operator $q\in\{\union{Y},\inter{Y}\}$, let $\dim(q)\defeq Y$.
\end{definition}
The operators $\inter{Y},\union{Y}$ are extensively used under many different notations in the literature. Most works noted the duality between these opersators. However, our notations highlights the fact that this duality stems from the famous duality of Boolean algebra. For example, several works (\cite{Dress2}- page 674,(11) and \cite{Welzl}- observation 22) observed the following lemma, Lemma \ref{lem:de-morgan}, but didn't notice that this lemma is the famous De-morgan laws.
\begin{lemma}[De-Morgan]\label{lem:de-morgan}
	Let $Y$ be a set, then:
	\begin{itemize}
	\item{$\inter{Y}\circ\neg=\neg\circ\union{Y}$}
	\item{$\union{Y}\circ\neg=\neg\circ\inter{Y}$}
	\end{itemize}
\end{lemma}

\subsubsection{Analogues of theorems and lemmas}
\begin{definition}\label{def:derivatives on x}
Let $\s$ be a system and let $x\in \dim(\s)$. The \deftext{derivatives of $\s$ associated with $x$},are the set $\{\union{\{x\}}(\s),\inter{\{x\}}(\s)\}$.
\end{definition}
Several lemmas and theorems from chapter $2$ contain the phrase ``let $\{\s',\s''\}$ be the pair of restrictions of $\s$ associated with $x$''. Namely: Lemma \ref{lem:str and oplus} and Theorem \ref{thm:RRR}. The only properties of the above $\s',\s''$ that were used in proving these lemmas and theorems are: $\lvert\s'\rvert+\lvert\s''\rvert=\lvert\s\rvert$ and the fact that $\{\s',\s''\}$ satisfy Lemma \ref{lem:str and oplus}. It is easy to verify that the pair of derivatives (associated with $x$) of a system $\s$ also have these two properties. Hence these theorems and lemmas hold when the above phrase is replaced with the phrase ``let $\{\sysnot{D}',\sysnot{D}''\}$ be the two derivatives of $\s$ associated with $x$''. Hence, the Sandwich Theorem may be proved by induction via the derivatives rather than the restrictions. An important result that is established this way is the analogue of Theorem \ref{thm:RRR}:
\begin{theorem}\label{thm:QQQ}
Let $\s$ be a system, let $x\in \dim(\s)$ and let $\sysnot{D}',\sysnot{D}''$ be the pair of derivatives of $\s$ associated with $x$. Then the following statements are equivalent:
	\begin{enumerate}
	\item{$SE(\s)$.}
	\item{$SE(\sysnot{D}')$ and $SE(\sysnot{D}'')$ and $\str(\sysnot{D}')\oplus_x \str(\sysnot{D}'')=\str(\s)$.}
	\item{$SE(\sysnot{D}')$ and $SE(\sysnot{D}'')$ and $\sstr(\sysnot{D}')\oplus_x \sstr(\sysnot{D}'')=\sstr(\s)$.}
	\end{enumerate}
\end{theorem}
It is quite surprising that these claims concerning $\sysnot{D}',\sysnot{D}''$ are established using arguments that are symmetric in $\sysnot{D}'$ and $\sysnot{D}''$.

\subsubsection{$\inter{Y}$ and $\union{Y}$ preserve $SE$}
The following theorem was proven by several authors:
\begin{theorem}\label{thm:SE kept under bool}
Let $\s$ be a system and let $Y\subseteq \dim(\s)$, then
	\begin{itemize}
	\item{$\inter{Y}(\s)$ is SE}
	\item{$\union{Y}(\s)$ is SE}
	\end{itemize}
\end{theorem}
This theorem is easily follows from Theorem \ref{thm:QQQ} and the following simple fact.
Let $Y,Y'$ be disjoint sets. and let $\alpha$ be one of the symbols ``$\bigsyscap$'', ``$\bigsyscup$''. Then:
$$\alpha_Y\circ\alpha_Y'=\alpha_{Y\cup Y'}$$

\subsection{Down-Shifting}
The down shifting operator is often used in the context of shattering. We prefer to present this operator in the terminology of oblivious sorting algorithms (See \cite{Knuth}). To this end, we present a system $\s$ by its characteristic function and assume that every $v\in C(\s)$ has a bit $b_v$ that is either $1$ or $0$ according to whether $v\in S(\s)$ or $v\notin S(\s)$. Our ``algorithms'' transform the system $\s$ by changing these bits.
 
Note that a 1-dimensional cube contains two vertices. Therefore we refer to such cubes as {\it edges} and to an $x$-cube as an {\it $x$-edge}.

For a set $X$, we let $\smaller$ denote the partial order on $\{0,1\}^X$ that is the product of the order $0<1$.\\
\begin{definition}
We say that an edge $e$ is {\it sorted under a system $\s$} if: 
$$e=\{u',u''\},~u'\smaller u''\mbox{ and } b_{u'}\geq b_{u''}.$$
\end{definition}
Note that the order is reversed. This is not a typo; rather, it follows the standard definition of down shifting.

The most basic operation of our algorithms {\it sorts} an edge $e=\{u,v\}$ by permuting, if needed, the contents of $b_u,b_v$.
Note that this operation mimics a ``comparator'' (see \cite{Knuth}) of a sorting network.

\begin{definition}[Down-shifting on $x$]
Let $\s$ be a system and let $x\in \dim(\s)$. \deftext{Down shifting $\s$ on $x$} is the operation of sorting every $x$-edge. The resulting system is denoted $\dshift_x(\s)$.
\end{definition}
\begin{lemma}\label{lem:shifting, restrictions and derivatives}
Let $\s$ be a system and let $x\in \dim(\s)$. Let $\{\sysnot{D}',\sysnot{D}''\}$ be the pair of derivatives of $\s$ associated with $x$ and let $\{\sysnot{R}',\sysnot{R}''\}$ be the pair of restrictions of $\dshift_x(\s)$ associated with $x$. Then:
	$$
	\{\sysnot{R}',\sysnot{R}''\}=\{\sysnot{D}',\sysnot{D}''\}
	$$
\end{lemma}
\begin{definition}
Let $\s$ be a system.
\begin{itemize}
	\item{We say that $\s$ is \deftext{$x$-sorted } if every $x$-edge is sorted under $\s$.}
	\item{We say that $\s$ is \deftext{edge-sorted} if every edge is sorted under $\s$.}
\end{itemize}
\end{definition}
The following lemma is easy:
\begin{lemma}\label{lem:oplus and str for monotone systems}
Let $\s$ be a system and let $x\in \dim(\s)$ such that $\s$ is $x$-sorted and let $\{\s',\s''\}$ be the pair of restrictions of $\s$ associated with $x$. Then:
	\begin{enumerate}
	\item{$\str(\s')\oplus_x \str(\s'')=\str(\s)$}
	\item{$\sstr(\s')\oplus_x \sstr(\s'')=\sstr(\s)$}
	\end{enumerate}
\end{lemma}

It is sometimes desired to translate our systems to the ``classical'' format of set-system. To this end:
\begin{definition}
Let $\s$ be a system.
	$$\set(\s)\defeq\{f^{-1}(1)~\vert~f\in S(\s)\}$$
\end{definition}

Observe that a system $\s$ is edge-sorted if and only if $\set(\s)$ is closed under the subset relation.
Using this observation is not hard to prove that:
\begin{lemma}
Let $\s$ be an edge-sorted system. Then:
$$\sstr(\s)=\str(\s)=\set(\s)$$
\end{lemma}

``Edge-sorting'' a system by applying a sequence of down-shifts is used often in the literature. The theory of sorting teaches us how many shiftings are required to make the system edge-sorted, as indicated in the next known lemma (See \cite{Knuth,Karp}):
\begin{lemma}
Let $A\in K^{m\times n}$ be a 2-dimensional array with keys from a linear ordered set $K$. If the rows of $A$ are sorted, then sorting the columns of $A$ preserves the rows sorted.
\end{lemma}
This fact is easily generalized to multi-dimensional arrays. In our context, it is manifested by the following lemma.
\begin{lemma}\label{lem:karp fact on systems}
Let $\s$ be a system and let $x,y\in \dim(\s)$, then $\dshift_x(\dshift_y(\s))$ is $y$-sorted.
\end{lemma}
\begin{definition}
Let $Q$ be a sequence of operators and let $x$ be an object. Let $Q(x)$ denote the result of applying the operators in $Q$ on $x$ one after the other.
\end{definition}
\begin{definition}\label{def:``sequence of''}
The phrase ``$Q$ is a sequence of $W$'' means that $W$ is a set and $Q$ is a sequence of members of $W$ in which each member appears exactly once.
\end{definition}
An immediate corollary of Lemma \ref{lem:karp fact on systems} is the following fact which was observed by several authors: 
\begin{theorem}
Let $\s$ be a system and let $Q$ be a sequence of $\{\dshift_x~:~x\in \dim(\s)\}$. Then $Q(\s)$ is edge-sorted.
\end{theorem} 
\subsubsection{Shifting preserves $SE$}
The following theorem was proven by several authors:
\begin{theorem}\label{thm:SE kept under shifting}
Let $\s$ be a system such that $SE(\s)$ and let $x\in \dim(\s)$. Then
	\begin{center}
	$\dshift_x(\s)$ is $SE$
	\end{center}
\end{theorem}
\begin{proof}
Let $\sysnot{R}'$ and $\sysnot{R}''$ denote the two restricitons of $\dshift_x(\s)$ associated with $x$.
By Theorem \ref{thm:RRR}, it is enough to prove that
	\begin{enumerate}
	\item{$SE(\sysnot{R}')$ and $SE(\sysnot{R}'')$,}
	\item{$\str(\sysnot{R}')\oplus_x \str(\sysnot{R}'')=\str(\dshift_x(\s))$.}
	\end{enumerate}
Statement 1 easily follows from Lemma \ref{lem:shifting, restrictions and derivatives} and Theorem \ref{thm:QQQ}.
Statement 2 follows from Lemma \ref{lem:oplus and str for monotone systems} 	
\end{proof}


\newpage
\section{Commutativity between Boolean operators}
Henceforth, ``Boolean operators'' means operators of the form $\union{Y}$ and $\inter{Y}$.
This chapter studies the commutativity between Boolean operators and its relations to the $SE$ property.

Since these operators are partial - it is important to note the following:
\begin{itemize}
	\item{Let $f:A\rightarrow B,g:A'\rightarrow B'$ be two operators. The operator $f\circ g$ is {\bf always} defined.\footnote{If $Range(g)\cap Dom(f)=\emptyset$ then $f\circ g$ is the operator which is nowhere defined (i.e $ Dom(f\circ g)=\emptyset$)}
	}
	\item{It is convenient that statements of the form ``$e_1 = e_2$'' are always meaningful, even when $e_1$ or $e_2$ are undefined. To this end the statement 
		``$e_1=e_2$'' is considered to be true either when both sides are defined and equal or when both are undefined. In any other case the statement is considered to be false.}
\end{itemize}

The following lemma summerizes the trivial cases regarding commutativity of the Boolean operators.
\begin{lemma}\label{lem:trivial commuting}
Let $Y,Y'$ be sets and let each of $\alpha,\beta$ be one of the symbols ``$\bigsyscap$'', ``$\bigsyscup$''. 
	\begin{enumerate}
	\item{If $Y\cap Y' \neq \emptyset$  then $\alpha_Y\circ\beta_{Y'} = \beta_{Y'}\circ\alpha_Y$ (both of them are nowhere defined).}
	\item{$\alpha_Y\circ \alpha_{Y'}  = \alpha_{Y'}\circ \alpha_Y$.  Moreover: if $Y\cap Y'=\emptyset$ then $\alpha_Y \circ \alpha_{Y'} = \alpha_{Y\cup Y'}$}
	\end{enumerate}	 
\end{lemma}
It is left to study the case of $Y\cap Y'=\emptyset$.
\subsection{Boolean operators commute on $SE$ systems}
Welzl et al have shown in \cite{Welzl} the following theorem:
\begin{theorem}
Let $\s$ be a maximum system and let $x,y\in \dim(\s)$ such that $x\neq y$. Then: $$\inter{\{x\}}\circ\union{\{y\}}(\s)=\union{\{y\}}\circ\inter{\{x\}}(\s)$$  
\end{theorem}
We show the following stronger lemma:
\begin{lemma}\label{lem:elementary commuting}
Let $~\s$ be a system and let $x,y\in \dim(\s)$ then:
	\begin{enumerate}
	\item{$\inter{\{x\}}\circ\union{\{y\}}(\s)\supseteq\union{\{y\}}\circ\inter{\{x\}}(\s)$}
	\item{$SE(\s)\implies\inter{\{x\}}\circ\union{\{y\}}(\s)=\union{\{y\}}\circ\inter{\{x\}}(\s)$}
	\end{enumerate}
\end{lemma}
Remark: Statement 1 was established by Welzl et al in \cite{Welzl}.
\begin{proof}
We consider two cases.

Case 1: $\dim(\s)=\{x,y\}$. There are exactly $16$ systems under this case and it is easy to verify that the lemma holds for all of them.

Case 2: $\lvert \dim(\s)\rvert>2$. Let $f=\inter{\{x\}}\circ\union{\{y\}}$, $g=\union{\{y\}}\circ\inter{\{x\}}$ and let $v\in \{0,1\}^{\dim(\s)-\{x,y\}}$.
It is easy to see that there exist an $\{x,y\}$-cube $C\subseteq C(\s)$ such that:
	$$
	f(\restr{\s}{C})=\restr{f(\s)}{\{v\}},
	$$
	$$
	g(\restr{\s}{C})=\restr{g(\s)}{\{v\}}
	$$
Thus, the lemma follows by Case 1 and Theorem \ref{thm:preservation under restrictions}.
\end{proof}
The following lemma is an easy corollary of the first conclusion of Lemma \ref{lem:elementary commuting}:
\begin{lemma}\label{lem:boolean hrcy}
Let $Q$ be a sequence of operators from $\{\inter{\{x\}},\union{\{x\}}~\vert~x\in \dim(\s)\}$ such that $Q(\s)$ is defined. Let $\overline{Q}$ be a permutation of $Q$ in which all the unions appear at the beginning and let $\underline{Q}$ be a permutation of $Q$ in which all the intersections appear at the beginning. Then:
	$$ 
	\underline{Q}(\s)\subseteq Q(\s)\subseteq\overline{Q}(\s)
	$$
\end{lemma}
\begin{definition}
Let $Q$ be a sequence of operators and let $x$ be an object. We say that $Q$ {\it commutes} on $x$ if for every $Q'$, a permutation of $Q$:
	$$
	Q(x)=Q'(x)
	$$
\end{definition}
From \ref{thm:SE kept under bool}, \ref{lem:trivial commuting} and the second conclusion of Lemma \ref{lem:elementary commuting} we get:
\begin{lemma}\label{lem:strong commuting}
If $~\s$ is an $SE$ system and $Q$ is a sequence of Boolean operators, then $Q$ commutes on $\s$.
\end{lemma}
Lemma \ref{lem:elementary commuting} by itself does not immediately imply \ref{lem:strong commuting}. Namely, having $\inter{\{x\}}$ and $\union{\{y\}}$ commute on $\s$ for each pair $x,y\in \dim(\s)$ does not imply that every sequence $Q$ of such operators commutes on $\s$. (See Theorem \ref{thm:refined bool commuting})\\

\subsection{A characterization of {\it Shattering} and {\it Strong-Shattering}}
To complete the picture, what is left is to study the other direction: when does commutativity of Boolean operators on $\s$ imply $SE(\s)$.  
To this end, we present the following characterization of shattering and strong-shattering:
\begin{theorem}\label{thm:str and sstr characterization by boolean operators}
Let $\s$ be a system and let $Y\subseteq X=\dim(\s)$. Then:
\begin{enumerate}
	\item{$\s$ shatters $Y$ $\iff$ $\inter{Y}\circ\union{X-Y}(\s)=\mathcal{K}_1$}
	\item{$\s$ strongly shatters $Y$ $\iff$ $\union{X-Y}\circ\inter{Y}(\s)=\mathcal{K}_1$}
\end{enumerate}
\end{theorem}
Recall that $\mathcal{K}_1$ is the non-empty $\emptyset$-system 

Theorem \ref{thm:str and sstr characterization by boolean operators} follows from the next two straightforward lemmas.
\begin{lemma}\label{bool char sstr sstr helper 1}
For every system $\s$, the following statements hold:
\begin{enumerate}
	\item{$\union{\dim(\s)}(\s)=\mathcal{K}_1\iff S(\s)\neq\emptyset$}
	\item{$\inter{\dim(\s)}(\s)=\mathcal{K}_1\iff S(\s)=C(\s)$}
\end{enumerate}
\end{lemma}
\begin{lemma}\label{bool char sstr sstr helper 2}
For every system $\s$ and $Y\subseteq \dim(\s)$:
\begin{enumerate}
	\item{$Y\in \str(\s)\iff S(\union{\dim(\s)-Y}(\s))=C(\union{\dim(\s)-Y}(\s))$}
	\item{$Y\in \sstr(\s)\iff S(\inter{Y}(\s))\neq\emptyset$}
\end{enumerate}
\end{lemma}

\subsection{A characterization of $SE$}
Our characterization in Theorem \ref{thm:str and sstr characterization by boolean operators} combined with Lemma \ref{lem:strong commuting} gives a new and simple proof of the following theorem.
\begin{theorem}[\cite{Dress2}]\label{thm:bool commuting}
The following statements are equivalent for any system $\s$:
	\begin{enumerate}
	\item{$SE(\s)$}
	\item{$\forall Y',Y''\subseteq \dim(\s)$: $\inter{Y'},\union{Y''}$ commute on $\s$}
	\end{enumerate}
\end{theorem}

\newpage
\section{Commutativity between Shifting operators}
\subsection{Shifting operators commute on $SE$ systems}
\begin{lemma}\label{lem:elementary commuting - shifting}
Let $~\s$ be a $SE$ system and let $x,y\in \dim(\s)$ such that $x\neq y$. Then	$\dshift_x$ and $\dshift_x$ commute on $\s$. 
\end{lemma}
\begin{proof}
We consider two cases.

Case 1: $\dim(\s)=\{x,y\}$. There are exactly $16$ systems under this case and it is easy to verify that the lemma holds for all of them.

Case 2: $\lvert \dim(\s)\rvert>2$. let $f=\dshift_x\circ\dshift_y$, $g'=\dshift_y\circ\dshift_x$. It is enough to prove that for every $\{x,y\}$-cube, $C\subseteq C(\s)$: 
$$\restr{f(\s)}{C}=\restr{g(\s)}{C}$$
It is easy to see that for every such $C$:
	$$
	\restr{f(\s)}{C}=f(\restr{\s}{C}),
	$$
	$$
	\restr{g(\s)}{C}=g(\restr{\s}{C})
	$$
Thus, the desired equality follows by Case 1 and Theorem \ref{thm:preservation under restrictions}.
\end{proof}
From \ref{thm:SE kept under shifting} and Lemma \ref{lem:elementary commuting - shifting} we get:
\begin{lemma}\label{lem:strong commuting - shifting}
If $~\s$ is an $SE$ system, and $Q$ is a sequence of operators from $\{\dshift_x~\vert~x\in \dim(\s)\}$ then $Q$ commutes on $\s$.
\end{lemma}

\subsection{A characterization of {\it Shattering} and {\it Strong-Shattering}}
In this section we present a characterization of $\str$ and $\sstr$ in terms of shiftings.
Recall from Definition \ref{def:``sequence of''} that he phrase ``$Q$ is a sequence of $W$'' means that $W$ is a set and $Q$ is a sequence of members of $W$ in which each member appears exactly once.
\begin{definition}
For a set $X$, define:
	$$
	\fshift(X)\defeq\big\{Q~:~Q\mbox{ is a sequence of }\{dshift_x~:~x\in X\}\big\}
	$$
\end{definition}
The following theorem is the main theorem of this section.
\begin{theorem}\label{thm:str,sstr characterization - shifting}
For every system $\s$:
\begin{enumerate}
\item{$$\str(\s)=\set(\bigsyscup\{Q(\s)~:~Q\in \fshift(\dim(\s))\})$$}
\item{$$\sstr(\s)=\set(\bigsyscap\{Q(\s)~:~Q\in \fshift(\dim(\s))\})$$}
\end{enumerate}
\end{theorem}
The following definition will be useful:
\begin{definition}
Let $x$ be an object and $i\in\{0,1\}$. Define:
$$
	\beta(x,i)\defeq
	\begin{cases}
		\union{\{x\}} &\mbox{if }i=0\\
		\inter{\{x\}} &\mbox{if }i=1
	\end{cases}
$$
\end{definition}
\begin{definition}
Let $\s$ be a system and let $x\in \dim(\s)$.\\
Let $C_i\defeq\{v\in C(\s)~:~v(x)=i\}$, where $i\in\{0,1\}$ denote the two sub-cubes of $C(\s)$ associated with $x$.\\
Define:
	$$
	\restr{\s}{x=i}\defeq\restr{\s}{C_i},~~i\in\{0,1\}
	$$
\end{definition}
The following lemma is straightforward.
\begin{lemma}\label{lem: down shifting and bool operations singleton}
Let $x,y$ be distinct elements and let $i,j\in\{0,1\}$. Then:
	\begin{enumerate}
	\item{The two operators $\restr{}{x=i}$ and $\restr{}{y=j}$ commute.}
	\item{The two operators $\restr{}{x=i}$ and $\dshift_y$ commute}
	\item{$\restr{}{x=i}\circ\dshift_x=\beta(x,i)$}
	\end{enumerate}
\end{lemma}
\begin{definition}
Let $X$ be a set, let $Q\in \fshift(X)$ and let $v\in\{0,1\}^X$. Define $Q_{[v]}$ to be the sequence of Boolean operators that is derived from $Q$ by replacing every $\dshift_x$ by $\beta(x,v(x))$. 
\end{definition}
\begin{lemma}\label{lem:shifting characterization helper}
Let $\s$ be a system, let $v\in C(\s)$ and let $Q\in \fshift(\dim(\s))$. Then:
$$\restr{Q(\s)}{\{v\}}=Q_{[v]}(\s)$$
\end{lemma}
\begin{proof}
Let $X=\dim(\s)$ and let $Z$ be a sequence of $\{\restr{}{x=v(x)}\:~:~x\in X\}$. It is easy to see that for every $X$-system $\sysnot{G}$:
	$$
	Z(\sysnot{G})=\restr{\sysnot{G}}{\{v\}}.
	$$
Thus, $Z(Q(\s))=\restr{Q(\s)}{\{v\}}$. Also, by Lemma \ref{lem: down shifting and bool operations singleton}:
	$$
	Z\circ Q = Q_{[v]}.
	$$
Therefore, $Z(Q(\s))=Q_{[v]}(\s)$. This implies the desired equality.
\end{proof}
The following two lemmas finish the proof of Theorem \ref{thm:str,sstr characterization - shifting}:
\begin{lemma}
For every system $\s$:
$$\str(\s)=\set(\bigsyscup\{Q(\s)~:~Q\in \fshift(\dim(\s))\})$$
\end{lemma}
\begin{proof}
Let $v\in C(\s)$. It is sufiicient to prove that:
	$$
	v^{-1}(1)\in \str(\s)\iff \restr{\bigsyscup\{Q(\s)~:~Q\in \fshift(\dim(\s))\}}{\{v\}}=\mathbb{K}_1.
	$$
Indeed:
\newpage
	\begin{align*}
	\restr{\bigsyscup\{Q(\s)~:~Q\in \fshift(\dim(\s))\}}{\{v\}} &= \intertext{by the definition of $\bigsyscup$}
	\bigsyscup\{\restr{Q(\s)}{\{v\}}~:~Q\in \fshift(\dim(\s))\} &= \intertext{by Lemma \ref{lem:shifting characterization helper}}
	\bigsyscup\{Q_{[v]}(\s)~:~Q\in \fshift(\dim(\s))\}				   &= \intertext{by Lemma \ref{lem:boolean hrcy}}
	\inter{v^{-1}(1)}(\union{v^{-1}(0)}(\s)) 										      
	\,.
	\end{align*}
By Theorem \ref{thm:str and sstr characterization by boolean operators} the lemma follows.
\end{proof}
Using similar arguments we also prove that:
\begin{lemma}
For every system $\s$:
$$\str(\s)=\set(\bigsyscap\{Q(\s)~:~Q\in \fshift(\dim(\s))\})$$
\end{lemma}

\subsection{A characterization of $SE$}
Theorem \ref{thm:SE kept under shifting} and Theorem \ref{thm:str,sstr characterization - shifting} imply the following theorem:
\begin{theorem}[\cite{BR95,Dress2}]\label{thm:SE characterization shifting}
The following two properties of a system $\s$ are equivalent:
	\begin{enumerate}
	\item{$\s$ is $SE$}
	\item{For every $Q',Q''\in \fshift(\s)$: 
	$$Q'(\s)= Q''(\s)$$}
\end{enumerate}
\end{theorem}  

\newpage
\section{A weaker form of Shattering Extremality}
This chapter introduces a weaker form of Shattering Extremality which is denoted by {\it $k$-Shattering Extrmality} (or $k$-$SE)$ where $k$ is a natural number. We study this property via the concept of {\it local-operators} which is interesting in its own right. The first section introduces the concept of local operators and discusses some basic properties of these operators. The second section defines $k$-$SE$ and proves generalizations of Theorems \ref{thm:SE characterization shifting} and \ref{thm:bool commuting} regarding equivalence of $k$-$SE$ with weaker forms of commutativity of Boolean operators and of down-shiftings operators.

\subsection{Local operators}
Let $SYS$ denote the set of all systems. For a  system $\s$ and $Y\subseteq \dim(\s)$, the collection of $Y$-cubes of $C(\s)$ forms a partition of $C(\s)$. This partition induces a partition of $\s$ in which every component corresponds to a restriction of $\s$ to a $Y$-cube. Informally, a {\it $Y$-local operator} is a mapping from $SYS$ to $SYS$ whose behaviour on $\s$ depends only on its behaviour on the restrictions of $\s$ to $Y$-cubes. Surprisingly, all the operators discussed in the previous chapters ($\union{Y},\inter{Y},\dshift_x,\restr{}{x=i}$ and $\neg$) are ``highly-local'' (this will become formal later). 

Let $X$ be a set. It is convenient here to extend the definition of ``$C$ is a $Y$-cube of $\{0,1\}^X$'' to the case where $Y\not\subseteq X$. So, a new {\it $Y$ cube of $\{0,1\}^X$} when $Y\not\subseteq X$ is an $X\cap Y$-cube of $\{0,1\}^X$ under the original terminology. It is also useful to extend the definition of the operators $\restr{}{x=0}$ and $\restr{}{x=1}$ in the following way: For a system $\s$ such that $x\notin \dim(\s)$, define $\restr{\s}{x=0}=\restr{\s}{x=1}=\sys{\emptyset}{\{0,1\}^{\dim(\s)}}$. This way, these operators are total operators from $SYS$ to $SYS$.

\begin{definition}
Let $Y$ be a set. An operator $\alpha:SYS\rightarrow SYS$ is called $Y$-local if the following statements hold:
	\begin{enumerate}
	\item{$\alpha$ is total.}
	\item{There exists a set $Z\subseteq Y$ such that for every system $\s$: $\dim(\alpha(\s))=\dim(\s)-Z.$}
	\item{$\alpha$ commutes with the operators $\restr{}{x=0}$ and $\restr{}{x=1}$ for all $x\notin Y$.}
	\end{enumerate}
\end{definition}         
The reader might wonder how does this definition captures the concept of locality. To answer this question, assume, for simplicity, that $Y\subseteq \dim(\s)$ and that $Z=\emptyset$ (namely, for all $\s$: $\dim(\alpha(\s))=\dim(\s))$. Note that every restriction of $\s$ to a $Y$-cube can be described by an appropriate sequence of operators of $\{\restr{}{x=i}~:~x\in \dim(\s)-Y\}$. Thus, item $3$ in the definition of $Y$-local amounts to:
\begin{center}
For every $C$, a $Y$-cube of $C(\s)$: $\restr{\alpha(\s)}{C}=\alpha(\restr{\s}{C})$.
\end{center}
Therefore, each component in the partitioning of $\alpha(\s)$ to the $Y$-cubes of $C(\s)$ depends only on the corresponding component of $\s$. Thus, the local behaviour of $\alpha$ on $Y$-systems determines its behaviour on every system. The fact that every $Y$-local operator is determined by its behaviour on $Y$-systems is expressed in the following straightforward lemma.
\begin{lemma}\label{lem:Y-cubes determine Y-local ops}
Let $\alpha$ and $\beta$ be two $Y$-local operators and let $\s$ be a system. Then the following statements are equivalent:
	\begin{enumerate}
	\item{$\alpha(\s)=\beta(\s)$}
	\item{For every $Y$-cube of $C(\s)$, $C$: $\alpha(\restr{\s}{C})=\beta(\restr{\s}{C})$}
	\end{enumerate}
\end{lemma}
\begin{example}
The operator $\neg$ and the Identity operator are $\emptyset$-local with $Z=\emptyset$.
\end{example}
\begin{example}
Let $x$ be an object. The operators $\restr{}{x=0}$ and $\restr{}{x=1}$ are $\{x\}$-local with $Z=\{x\}$.
\end{example}
\begin{example}\label{Example:locality shift}
The operator $\dshift_x$ is not total since it is not defined on systems $\s$ with $x\notin \dim(\s)$. Thus, henceforth this operator is extended to such systems by $\dshift_x(\s)=\sys{\emptyset}{\{0,1\}^{\dim(\s)}}$. Under this modification, $\dshift_x$ is $\{x\}$-local with $Z=\emptyset$.
\end{example}
\begin{example}\label{Example:locality bool}
The operator $\alpha_{\{x\}}$ where $\alpha$ is one of the symbols `$\bigsyscup$' or `$\bigsyscap$', is not defined on systems $\s$ for which $x\notin \dim(\s)$. Thus, henceforth this operator is extended to such systems by $\alpha_{\{x\}}(\s)=\sys{\emptyset}{\{0,1\}^{\dim(\s)}}$. Under this modification, $\alpha_{\{x\}}$ is $\{x\}$-local with $Z=\{x\}$.
\end{example}
The following lemma is straightforward.
\begin{lemma}\label{lem:Y-locality main lemma}
Let $X$ and $Y$ be sets and let $\alpha$ be an $X$-local operator and $\beta$ be a $Y$-local operator. Then $\alpha\circ\beta$ is $X\cup Y$-local. 
\end{lemma}
Using Lemma \ref{lem:Y-locality main lemma} we can compute the locality of sequences of local operators.

Let $Q$ be a sequence of Boolean operators. We say that $Q$ is {\it meaningful} if every two distinct operators, $q',q''$, in $Q$ satisfy\footnote{Recall from Definition \ref{def:bool ops} that $\dim(\inter{Y})=\dim(\union{Y})=Y$.} $\dim(q')\cap \dim(q'')=\emptyset$. Note that if $Q$ is a sequence of Boolean opearators which is not meaningful then the composition of the elements of $Q$ is the operator which is nowhere defined. Thus, all the theorems and lemmas we have proved regarding sequences of Boolean operators remain true if we consider only meaningful sequences.  For a sequence of Boolean operator $Q$, define 
\[
\dim(Q)\defeq\bigcup_{q\mbox{ is in }Q}{\dim(q)}
\]
Thus, by Lemma \ref{lem:Y-locality main lemma} we obtain:
\begin{example}\label{Example:locality bool sequences}
Let $Q$ be a meaningful sequence of Boolean operators. Then $Q$ is $\dim(Q)$-local.
\end{example}
Similarly, we have for down-shifting operators:
\begin{example}\label{Example:locality shift sequences}
Let $X$ be a set and let $Q\in \fshift(X)$. Then $Q$ is $X$-local.
\end{example}

\subsection{$k$-$SE$}
\begin{definition}\label{def:k-SE}
Let $\s$ be a system and let $k\in\mathbb{N}$. $\s$ is called $k-SE$ if:
\begin{center}
For every cube of $C(\s)$, $C$ with $\lvert \dim(C)\rvert\leq k$: $SE(\restr{\s}{C})$
\end{center}
We will refer to the statement ``$\s$ is $k-SE$'' by ``$SE_{k}(\s)$''
\end{definition}
\subsubsection{k-SE and Boolean operators}
The following theorem is the generalization of Theorem \ref{thm:bool commuting} concerning $k$-$SE$.
\begin{theorem}\label{thm:refined bool commuting}
Let $\s$ be a system and let $k\in \mathbb{N}$. The following statements are equivalent:
\begin{enumerate}
\item{Every meaningful sequence of Boolean operators $Q$ with $\lvert \dim(Q)\rvert\leq k$ commute on $\s$.}
\item{For all disjoint $Y',Y''$ such that $\lvert Y'\cup Y''\rvert \leq k$: $\inter{Y'}$ and $\union{Y''}$ commute on $\s$.}
\item{$SE_{k}(\s)$.}
\end{enumerate}
\end{theorem}
\begin{proof}
$1\iff 2$ is a simple corollary of Lemma \ref{lem:boolean hrcy}.

$2\iff 3$: By Lemma \ref{lem:Y-cubes determine Y-local ops} and Example \ref{Example:locality bool sequences}, Item $2$ is equivalent to the following statement.
	\begin{center}
 	For every $C$, a $X$-cube of $C(\s)$ where $\lvert X\rvert\leq k$: $\inter{Y}$ and $\union{X-Y}$ commute on $\restr{\s}{C}$ 
	\end{center}
By Theorem \ref{thm:bool commuting}, the above is equivalent to $SE_k(\s)$.
\end{proof}

\subsubsection{$k$-$SE$ and down-shifting operators}
The following theorem generalizes Theorem \ref{thm:str,sstr characterization - shifting}.
\begin{theorem}[Refined commution of down-shifts]
Let $\s$ be a system and let $k\in \mathbb{N}$. The following statements are equivalent:
\begin{enumerate}
\item{For all $Y\subseteq \dim(\s)$ such that $\lvert Y\rvert\leq k$ and for every $Q',Q''\in \fshift(Y)$: 
$$Q'(\s)= Q''(\s)$$}
\item{$SE_{k}(\s)$}
\end{enumerate}
\end{theorem}
\begin{proof}
The proof easily follows Lemma \ref{lem:Y-cubes determine Y-local ops}, Theorem \ref{thm:str,sstr characterization - shifting} and Example \ref{Example:locality shift sequences}
\end{proof}  

\newpage
\section{Arrangements of hyperplanes in a Euclidean space}
Let $H$ be an arrangement of hyperplanes in a Euclidean space, $V=\mathbb{R}^{d}$. For each hyperplane $e\in H$ let one half-plane determined by $e$ be its {\it positive side} and the other its {\it negative side}. The hyperplanes of $H$ cut $V$ into open regions ({\it cells}). Corresponding to each cell $c$ there is $f_{c}\in\{+,-\}^{H}$ such that
$$f_{c}(e)=\begin{cases}
+ & \mbox{if }c\mbox{ is in the positive side of }e\,,\\
- & \mbox{if }c\mbox{ is in the negative side of }e\,.
\end{cases}$$

In this chapter we analyze the system $\s_H\defeq\sys{\{f_c~:~c\mbox{ is a cell of }H\}}{\{+,-\}^H}$.
Section \ref{subsection:Notations and preliminaries} introduces preliminaries. Sections \ref{subsection:geom interpretation of str} and \ref{subsection:geom interpretation of sstr} prove the following characterizations of $\str(\s_H)$ and $\sstr(\s_H)$:
\begin{enumerate}[label=(\roman*)]
\item{$X\subseteq H$ is shattered by $\s_H$ if and only if the normal vectors of the hyperplanes in $X$ are linearly independent}\label{geom:int:item: i}
\item{$X$ is strongly-shattered by $\s_H$ if and only if in addition to \ref{geom:int:item: i} there is no hyperplane $h\in E-X$ such that $\bigcap{X}\subseteq h$.}\label{geom:int:item: ii}
\end{enumerate}

Note that if $H$ is an arrangement of lines in the plane and $\{l_1,l_2\}\subseteq H$ then the first property amounts to "$l_1$ and $l_2$ are not parallel" and the second property amounts to "The intersection point of $l_1,l_2$ is simple\footnote{there is no other line in $H$ that is incident to this point} in $H$".

Section \ref{subsection:counting cells} demonstrates a usage of the above characterizations of $\str(\s_H),\sstr(\s_h)$ for counting the number of cells of some arrangements.
Section \ref{subsection:Questions} discusses a certain generalization of such systems and presents some open questions.

Figure \ref{fig:hyperplanes} gives an example of an arrangement of 4 lines which partition the plane into $10$ regions. The positive side of each line in the arrangement is shaded.
\begin{figure}[!htb]
\centering
\includegraphics[scale=.7]{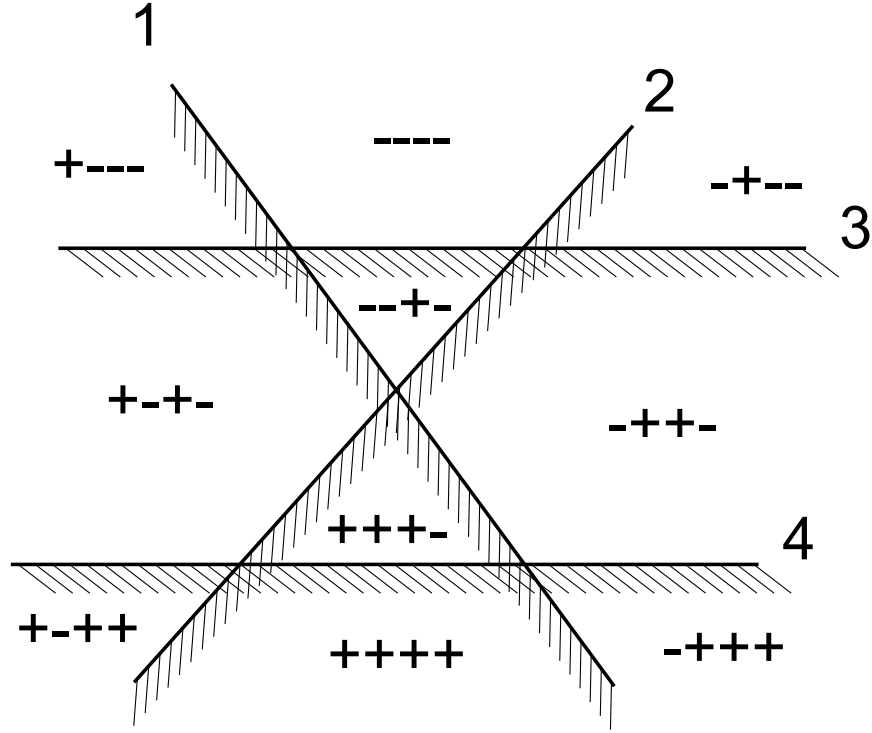}
\caption{An arrangement of (oriented) lines in the plane.}
\label{fig:hyperplanes}
\end{figure}
\newpage

\subsection{Notations and preliminaries}\label{subsection:Notations and preliminaries}
Let $V=\mathbb{R}^d$ be a Euclidean space. A {\it linear hyperplane}, is a $(d-1)$-dimensional subspace of $V$, i.e.,
$$
\{\vec{v}\in V~:~\vec{\alpha}\cdot\vec{v}=0\}
$$
where $\vec{\alpha}$ is a fixed nonzero vector in $V$ and $\vec{\alpha}\cdot\vec{v}$ is the usual inner product. An {\it affine hyperplane}, or simply a {\it hyperplane}, is a translate of a linear hyperplane, i.e.
$$
\{\vec{v}\in V~:~\vec{n}\cdot\vec{v}=r\}
$$
where $\vec{n}$ is a fixed nonzero vector in $V$ and $r$ is a fixed scalar. $\vec{n}$ is called a {\it normal} of the hyperplane. Note that the normal is unique up to scaling.

An (open) {\it half-space} is a set $\{\vec{v}\in V~:~\vec{\alpha}\cdot\vec{v}> c\}$  for some $\alpha\in V$, $c\in\mathbb{R}$. If $e$ is a hyperplane in $V$, then the complement $V-e$ has two components each of which is an open half-space.
It is convenient to associate every (oriented) hyperplane $e\subseteq V$, with a vector $\vec{n_e}\in V$ and a number $r_e\in\mathbb{R}$ such that:
	\begin{enumerate}
	\item{$e=\{\vec{x}\in V~:~\vec{n_e}\cdot \vec{x}=r_e\}$}
	\item{The {\it positive side} of $e$ is the set $\{\vec{x}\in V~:~\vec{n_e}\cdot \vec{x}>r_e\}$.}
	\end{enumerate}

For a vector $\vec{v}\in V$, let $\sign(\vec{v})\in\{+,0,-\}^d$ be such that, for all $i\leq d$: 
\[ \sign(\vec{v})_i\defeq \begin{cases}
						+ &\mbox{if }v_i>0\,,\\
						0 &\mbox{if }v_i=0\,,\\
						- &\mbox{if }v_i<0\,.
						\end{cases}
\]  
For $A\subseteq V$, let $ch(A)$ denote the {\it convex hull} of $A$.

Let $H$ be an arrangement of hyperplanes in $V$. We use $\bigcap{H}$ and $\bigcup{H}$ to denote the intersection and union of the hyperplanes in $H$.
\begin{definition}
$H$ is called \deftext{independent (dependent)} if the set $\{\vec{n_e}~:~e\in H\}$ is an independent (dependent) set of vectors.
\end{definition}
\begin{lemma}
If $H$ is independent then $\bigcap{H}\neq\emptyset$ 
\end{lemma}
\begin{proof}
Let $\vec{x}\in V$. By the definition of $\vec{n_e}$: 
$$\vec{x}\in\bigcap{H}\iff \forall e\in H: \vec{n_e}\cdot \vec{x}= r_e.$$
Thus it is enough to prove that there exists a solution, $\vec{x}$, for the linear system:
$$n_e\cdot x = r_e,~\forall e\in H.$$
Independence of the $n_e$'s implies that there exists such an $x$.
\end{proof}
\begin{definition}
Let $Y\subseteq H$. $Y$ is called \deftext{$H$-regular} if 
	$$Y=\{e\in H~:~\bigcap{Y}\subseteq e\}.$$ 
\end{definition}
\begin{lemma}\label{lem:irregularity}
Let $Y\subseteq H$ such that $\bigcap{Y}\neq\emptyset$ and let $h$ be a hyperplane such that $\bigcap{Y}\subseteq h$. Then
$$\vec{n_h}\in \spn\{\vec{n_e}~:~e\in Y\}.$$
\end{lemma}
\begin{proof}
Let $U\defeq \spn\{\vec{n_e}~:~e\in Y\}$. It is easy to verify that $\vec{n_h}\in (U^{\perp})^{\perp}=U$.
\end{proof}
\begin{lemma}\label{lem:regularity}
Let $Y\subseteq E$ be independent and $H$-regular. Then: 
	$$\bigcap{Y}-\bigcup{(H-Y)}\neq\emptyset.$$
\end{lemma}
\begin{proof}
By Lemma \ref{lem:irregularity}, there exists no $e\in H-Y$ such that $\bigcap{Y}\subseteq e$. Thus, every $e\in H-Y$ intersects $\bigcap{Y}$ in a proper affine sub-space, $A_e$ , of $\bigcap{Y}$. Thus, $\{A_e~:~e\in H-Y\}$ is a finite collection of proper affine sub-spaces of $\bigcap{Y}$ and thus, by Baire Theorem, cannot cover it.
\end{proof}

\subsection{Geometric interpretation of shattering}\label{subsection:geom interpretation of str}
Let $H$ be an arrangement of hyperplanes in $\mathbb{R}^{d}$ and let $\s_H$ be the corresponding system.\\
This section proves and discusses the following theorem:
\begin{theorem}\label{thm:main thm geometring str}
Let $X\subseteq H$. Then the following statements are equivalent
\begin{enumerate}
\item{$\s_H$ shatters $X$.}
\item{$X$ is independent.}
\end{enumerate}
\end{theorem}
\begin{lemma}
Let $X\subseteq H$. Then
\begin{center}
 $X$ is independent $\implies$ $\s_H$ shatters $X$.
\end{center}
\end{lemma}
\begin{proof}
It is enough to show that for every $f\in\{+,-\}^X$ there exists some cell $c$ such that $f_c$ and $f$ agree on $X$. Let $f\in\{+,-\}^X$. Define 
$$y_e\defeq \begin{cases} r_e + 1 &\mbox{if }f(e)=+\\
													r_e - 1 &\mbox{if }f(e)=-
						\end{cases}$$
Since $X$ is independent the linear system 
	$$
	\vec{n_e}\cdot \vec{x} = y_e, e\in X
	$$
has a solution in $\mathbb{R}^d$. Let $c$ denote a cell whose closure, $\overline{c}$, contains such a solution. Then $f_c$ and $f$ agrees on $X$ as required.
\end{proof}
Proving the other requires more work. We split it to two parts; first, it is shown that if $X\subseteq H$ is shattered by $\s_H$ then $\bigcap{X}\neq\emptyset$. In the second part we establish that if $\s_H$ shatters $X$ then $X$ is independent.

For the first part, the following lemma is useful.
\begin{lemma}\label{lem:center is in ch}
Let $k\in\mathbb{N}$ and let $A\subseteq\mathbb{R}^k$ be a set that satisfies:
	\[(\forall f\in\{+,-\}^k)(\exists v\in A):~\sign(v)=f.
	\]
Then $\vec{0}\in ch(A).$
\end{lemma}
Lemma \ref{lem:center is in ch} is easily proved by induction on $k$.
\begin{lemma}\label{lem:str implies intersection}
Let $X\subseteq H$ be a non-empty set. Then
$$X\in \str(\s_H)\implies \bigcap{X}\neq\emptyset.$$
\end{lemma}
\begin{proof}
Let $A\in\mathbb{R}^{X\times d}$ be a matrix with the $\vec{n_e}$'s in $X$ as its rows and let $\vec{r}\in\mathbb{R}^{X}$ be a vector with the $r_e$'s as its entries.
It is enough to show that there exists $x$ such that $A\vec{x}=\vec{r}$.

Let $T$ denote the map $\vec{x}\mapsto A\vec{x}-\vec{r}$. $T$ is an affine map. Since $\s_H$ shatters $X$, it follows that:
	$$
	(\forall f\in\{+,-\}^X)(\exists x\in\mathbb{R}^d):\sign(A\vec{x}-\vec{r})=f
	$$
Let $A\subseteq\mathbb{R}^d$ be a set that for every $f\in\{+,-\}^X$ contains $x_f\in\mathbb{R}^d$ for which $\sign(A\vec{x}-\vec{r})=f$. Therefore, according to Lemma \ref{lem:center is in ch}, it follows that $\vec{0}\in ch(T(A))$. Since $T$ is affine, we have that $ch(T(A))=T(ch(A))$. Thus, there exists $\vec{x}\in ch(A)$ such that $T(\vec{x})=0$, which means that: $A\vec{x}=\vec{r}$ as required.
\end{proof}
\begin{lemma}
Let $X\subseteq H$. Then
\begin{center}
 $\s_H$ shatters $X$ $\implies$ $X$ is independent.
\end{center}
\end{lemma}
\begin{proof}
By Lemma \ref{lem:str implies intersection} we may assume that for all $e\in X$: $r_e=0$. (Otherwise, translate $\mathbb{R}^d$ such that the origin is in $\bigcap X$)
Let $\{\alpha_e~:~e\in X\}$ be coefficients such that $\sum_{e\in X}{\alpha_e \vec{n_e}}=0$. It is enough to show that $\forall e\in X:\alpha_e=0$. 

Let $f\in\{+,-\}^X$ such that
$$f(e)\defeq \begin{cases} + &\mbox{if }\alpha_e\geq 0\\
													 - &\mbox{otherwise }
						\end{cases}.$$
Since $\s_H$ shatters $X$ there exists $x\in\mathbb{R}^d$ such that 
$$\sign(\vec{x}\cdot \vec{n_e})=f(e)$$
Furhtermore, we may choose $\vec{x}$ such that $\forall e\leq X: \lvert \vec{x}\cdot \vec{n_e}\rvert\geq 1$.
	\begin{align*}
	0 &= \vec{x}\cdot \sum_{e\in X}{\alpha_e \vec{n_e}} &\mbox{since }\sum_{e\in X}{\alpha_e \vec{n_e}}=0\\
		&= \sum_{e\in X}{\alpha_e \vec{x}\cdot \vec{n_e}}\\
		&\geq \sum_{\alpha_e\geq 0}{\alpha_e}-\sum_{\alpha_e< 0}{\alpha_e} &\mbox{by definition of $\vec{x}$}\\
	\end{align*}
Therefore:
$$\sum_{\alpha_e\geq 0}{\alpha_e}-\sum_{\alpha_e< 0}{\alpha_e}=0.$$
This is only possible if $\forall e\in X:\alpha_e=0$. 
\end{proof}

\subsection{Geometric interpretation of strong-shattering}\label{subsection:geom interpretation of sstr}
Let $H$ be an arrangement of hyperplanes in $\mathbb{R}^{d}$ and let $\s_H$ be the corresponding system.
This section proves and discusses the following theorem:
\begin{theorem}\label{thm:main thm geometring sstr}
Let $X\subseteq H$. Then the following statements are equivalent
\begin{enumerate}
\item{$\s_H$ strongly-shatters $X$.}
\item{$X$ is independent and $H$-regular.}
\end{enumerate}
\end{theorem}
\begin{lemma}
Let $X\subseteq H$. Then
\begin{center}
 $X$ is independent and $H$-regular $\implies$ $\s_H$ strongly-shatters $X$. 
\end{center}
\end{lemma}
\begin{proof}
According to Lemma \ref{lem:regularity} there exists $\vec{p}\in\bigcap{X}-\bigcup{(H-X)}$. Therefore:
	\[
	\forall e\in H-X: \lvert \vec{n_e}\cdot \vec{p} - r_e\rvert> 0.
	\]
Define $g\in\{+,-\}^{H-X}$ to be:
	$$
	g(e)\defeq \sign(\vec{n_e}\cdot \vec{p} - r_e)
	$$
It is enough to show that for all $f\in\{+,-\}^X$ there exists a cell $c$ such that $f_c=g\star f$. Let $f\in\{+,-\}^{X}$. Define  
$$y_e\defeq \begin{cases} 1  &\mbox{if }f(e)=+\,,\\
													-1 &\mbox{if }f(e)=-\,.
						\end{cases}$$
Consider the linear system 
	\[
	\vec{n_e}\cdot \vec{x} = y_e\mbox{ for all } e\in X
	\]
Since $X$ is independent, this system has a solution $\vec{d}$.\\
Pick $\epsilon>0$ sufficiently small such that
	\[
	\forall e\in E-X: \sign(\vec{n_e}\cdot \vec{p} - r_e)=\sign(\vec{n_e}\cdot (\vec{p}+\epsilon\vec{d}) - r_e).
	\]
Let $c$ denote the cell that contains $\vec{p}+\epsilon \vec{d}$. By the definitions of $\epsilon$ and $\vec{d}$ it follows that $f_c = g\star f$ as required.
\end{proof}
For the other direction:
\begin{lemma}
Let $X\subseteq H$. Then
\begin{center}
 $\s_H$ strongly-shatters $X$ $\implies$ $\s_H$ is independent and $H$-regular.
\end{center}
\end{lemma}
\begin{proof}
We will prove the contra-position of this claim. If $X$ is dependent then $\s_H$ does not shatter $X$. In particular, this means that $\s_H$ doesnt strongly-shatter $X$. Otherwise, assume $X$ is independnet and $H$-irregular. Therefore, by Lemma \ref{lem:str implies intersection} $\bigcap{X}\neq\emptyset$. Moreover, we may assume that $\vec{0}\in\bigcap{X}$ (by translating the space if needed). 

Since $X$ independent and $H$-irregular, it follows that there exists $h\in H-X$ such that $\bigcap{X}\subseteq h$. Thus, by Lemma \ref{lem:irregularity}, $\vec{n_h}\in \spn\{\vec{n_e}~:~e\in X\}$. which means that there exist coefficients $\{\alpha_e\in\mathbb{R}~:~e\in X\}$ such that 
$$\sum_{e\in X}{\alpha_e n_e = n_h}.$$
Note that $\vec{0}\in\bigcap{X}\subseteq h$ implies that $r_h=0$. Assume, by way of contradiction that $\s_H$ strongly-shatters $X$.
Let $$g(e)\defeq \begin{cases} + &\mbox{if }\alpha_e\geq 0\\
													 - &\mbox{otherwise }
						\end{cases}$$
and $$-g(e)\defeq\begin{cases} + &\mbox{if }g(e) = -\\
													 		 - &\mbox{otherwise. }
						\end{cases}$$
Since $X$ is strongly shattered by $\s_H$, there exist two open cells $c',c''$ such that:
	\begin{enumerate}[label=(\roman*)]
	\item{$f_{c'}$ agrees with $g$ on $X$}
	\item{$f_{c''}$ agrees with $-g$ on $X$}
	\item{$f_{c'}$ agrees with $f_c''$ on $H-X$} \label{enum: sstr item 3}
	\end{enumerate} 
We will derive a contradiction to item \ref{enum: sstr item 3} by showing that $f_{c'}(h)\neq f_{c''}(h)$.

Let $\vec{x'}\in c'$ and $\vec{x''}\in c''$.
	\begin{align*}
	f_{c'}(h)&=\sign(\vec{n_h}\cdot \vec{x'})\\
					 &=\sign([\sum_{e\in X}{\alpha_e \vec{n_e}}]\cdot \vec{x'}) & \sum_{e\in X}{\alpha_e \vec{n_e} = \vec{n_h}}\\
					 &=\sign(\sum_{e\in X}{\alpha_e [\vec{n_e}\cdot \vec{x'}}])\\
					 &=+ 																					 & \mbox{by definition of } f_{c'}: \alpha_e [\vec{n_e}\cdot \vec{x'}]\geq 0\\
					 &\,.
	\end{align*}
Similarly we have $f_{c''}(h)= -$ and we are done.
\end{proof}

\subsection{Counting cells}\label{subsection:counting cells}
The following lemma is a corollary of Theorems \ref{thm:main thm geometring str}, \ref{thm:main thm geometring sstr} and Lemma \ref{lem:irregularity}.
\begin{lemma}\label{lem:SE geometrical char}
Let $H$ be an arrangement of hyperplanes in $\mathbb{R}^d$. Then, the following statements are equivalent.
	\begin{enumerate}
	\item{$SE(\s_H)$}
	\item{$\forall X\subseteq H$: $X$ is independent $\iff$ $\bigcap{X}\neq\emptyset$.}
	\end{enumerate}
\end{lemma}
Let $H$ be an arrangement of hyperplanes such that $SE(\s_H)$. By Theorem \ref{thm:SE char}, this means that $\lvert \s\rvert=\lvert \str(\s)\rvert$. Thus, by Lemma \ref{lem:SE geometrical char} and Theorem \ref{thm:main thm geometring str} we get the following result:
\begin{lemma}\label{lem:counting lemma}
Let $H$ be an arrangement of hyperplanes in $\mathbb{R}^d$ such that 
	\begin{center}
	$\forall X\subseteq H$: $X$ is independent $\iff$ $\bigcap{X}\neq\emptyset.$
	\end{center}
Let $C$ denote the number of cells in the arrangement. Then:
$$C=\lvert\{X\subseteq H~:~X\mbox{ is independent}\}\rvert.$$
\end{lemma}
in particular, if the arrangement $H$ is in a general position (i.e. every $d$ hyperplanes intersect in a single point and no $d+1$ hyperplanes intersect), then it is easy to verify that $H$ satisfies the premise of Lemma \ref{lem:counting lemma} and that $$\{X\subseteq H~:~X\mbox{ is independent}\}=\{X\subseteq H~:~\lvert X\rvert\leq d\}.$$ This proves the well known result regarding the number of cells, $C$, in such an arrangement. Namely $C=\sum_{i=0}^{d}{{\lvert X\rvert}\choose{i}}$.

\subsection{Convex systems and some questions}\label{subsection:Questions}
Let $\s$ be a system. $\s$ is called {\it Euclidean} if there exist a Euclidean space, $\mathbb{R}^{d}$, and an arrangement of hyperplanes $H$ such that $\s=\s_H$.

Let $H$ be an arrangement of hyperplanes in $\mathbb{R}^{d}$ and let $\s_H$ be the corresponding system. Let $e\in H$ and let $\s_H',\s_H''$ be the two restrictions of $\s_H$ associated with $x$. It is easy to see that $\s_H'\syscup\s_H''$ corresponds to the arrangement $H'=H-\{e\}$ and that $\s_H'\syscup\s_H''$ corresponds to the arrangement $H'$ (of one dimension smaller) obtained as the intersection of $e$ with the remaining hyperplanes. Thus, Euclidean systems are closed under Boolean operators. It is not hard to prove that Euclidean systems are not closed under restrictions. However, the following generalization enables the representation of $\s_H',\s_H''$ in a geometric setup.

Let $K$ be an open convex set in $\mathbb{R}^{d}$. Let $$\s_{H,K}=\sys{\{f_c~:~c\mbox{ is a cell of }H,~c\cap K\neq\emptyset\}}{\{+,-\}^H}.$$
It is easy to see that choosing $K$ to be one of the half-spaces determined by $e$ yields a representation of $\s_H'$ and $\s_H''$ as $\s_{H,K}$.

This setup generalizes the setups investigated by Lawerence in \cite{Law} and by Welzl et al in \cite{Welzl}. The systems studied in \cite{Law} are obtained by choosing  $H=\big\{\{\vec{x}~:~x_i=0\}~:~1\leq i\leq d\big\}$ and $K$ to be arbitrary and the systems studied in \cite{Welzl} are obtained by choosing $H$ to be any general-position arrangement and $K=\mathbb{R}^d$.

The following two theorems can be proved similarly to the way in which we proved theorems \ref{thm:main thm geometring sstr}, \ref{thm:main thm geometring str}
\begin{theorem}
Let $X\subseteq H$. Then the following statements are equivalent
\begin{enumerate}
\item{$\s_{H,K}$ shatters $X$.}
\item{$X$ is independent.}
\item{$\bigcap{X}$ and $K$ share a common point.}
\end{enumerate}
\end{theorem}
\begin{theorem}
Let $X\subseteq H$. Then the following statements are equivalent
\begin{enumerate}
\item{$\s_{H,K}$ strongly-shatters $X$.}
\item{$X$ is independent, $H$-regular.}
\item{$\bigcap{X}$ and $K$ share a common point.}
\end{enumerate}
\end{theorem}

The above characterizations of $\str,\sstr$ in this geometric setup yields a big class of $SE$ systems that rise from arrangements of hyperplanes in a Euclidean space. It is natural to ask whether every $SE$ system can be implemented in such a way. The following is a partial answer to this question.
\begin{theorem}
There exists an $SE$ system $\s$ such that there exists no arrangement of hyperplanes, $H$, such that $\s=\s_H$.
\end{theorem}
\begin{proof}
By Theorem \ref{thm:main thm geometring str}, for every $H$, it follows that $\str(S_H)$ is a matroid. It is easy to construct an $SE$-family $\s$ such that $\str(s)$ is not a matroid.
\end{proof}
However, the answer to the following question is not known.
\begin{question}
Let $\s$ be an $SE$ system. Is there an arrangement $H$ and a convex set $K$ such that $\s=\s_{H,K}$?
\end{question}


\newpage
\section{Orientations of an undirected graph}
This chapter presents a method which uses the tools we developed in order to prove certain equalities and inequalities involving undirected graphs and their orientations. Unlike most applications of these concepts, which use the "shattering" relation, this one demonstrates a usage of "strong-shattering". The method is demonstrated via two examples.

Henceforth, let $G=(V,E)$ be an arbitrary fixed undirected simple graph.
We represent an orientation of $E$ as a function $d:E\rightarrow\{+,-\}$. In order to enable this, pick a fixed orientation $\overrightarrow{E}$ of $E$.
Thus, a function $d:E\rightarrow\{+,-\}$ define an orientation of $E$, relative to $\overrightarrow{E}$, in the obvious way: If $d(e)=+$ then $e$ is oriented the same way as in $\overrightarrow{E}$. Else, if $d(e)=-$ then $e$ is oriented in the opposite way as in $\overrightarrow{E}$.
For an orientation $d$, let $G_d$ denote the directed graph obtained by orienting $E$ according to $d$.

For $X\subseteq E$, let $G_X\defeq(V,X)$ denote the undirected subgraph of $G$ with $X$ as its edges.

\subsection{Cyclic orientations}
The first example proves the following inequality:
\begin{theorem}\label{thm:first graphs example}
The number of orientations of $G$ that yield a directed cycle is at least the number of its subgraphs that have an undirected cycle.
\end{theorem}
Consider the system $\s\defeq\sys{\{d\in\{+,-\}^E~:~G_d\mbox{ has a directed cycle}\}}{\{+,-\}^E}$. The main step in this method is to characterize $\sstr(\s)$ or $\str(\s)$ (or both).
In this example, the following lemma characterize $\sstr(\s)$.
\begin{lemma}
Let $X\subseteq E$. Then $X\in \sstr(\s) \iff G_{E-X}\mbox{ has a cycle }$
\end{lemma}
\begin{proof}
If there exists a cycle, $C$, in $E-X$ then there exists an orientation of $E-X$, $d$, such that $C$ is oriented to a directed cycle. Clearly, every extension of $d$ to an orientation of $E$ contains this directed cycle. This means that $X\in \sstr(\s)$

For the other direction, assume that $E-X$ doesn't contain a cycle. To establish that $X\notin \sstr(\s)$, it is enough to show that:
$$(\forall d':(E-X)\rightarrow\{+,-\})(\exists d'':X\rightarrow\{+,-\}):G_{d'\star d''}\mbox{ has no directed cycle}.$$
Let $d':(E-X)\rightarrow\{+,-\}$. Since $E-X$ has no cycles, it follows that the orientation of $G_{E-X}$ according to $d'$ is a DAG whose edges form a pre-order, $P$, on $V$. Pick (by topological sorting) a linear order $L$ of $V$ that extends $P$. Let $d'':X\rightarrow\{+,-\}$ be the orientation of the remaining edges according\footnote{every edge is oriented towards the bigger vertex} to $L$. Clearly, the resulting orientation, $d'\star d''$, is a-cyclic. 
\end{proof}
Theorem \ref{thm:first graphs example}, easily follows from the Sandwich Theorem (Theorem \ref{thm:basic inequality}). 
It is easy to see that in some graphs the inequality is strict, as observed by taking $G$ to be a simple cycle. Indeed - for such $G$ there are two possible orientations that yield a directed cycle but only one $X\subseteq E$ such that $G_X$ has a cycle.

Consider the system $\neg\s$, namely the orientations of $G$ that yield an a-cyclic graph. By Lemma \ref{obs:dualStrSstr} it follows that
\begin{lemma}
Let $X\subseteq E$. Then $X\in \str(\neg\s) \iff G_{X}\mbox{ is a forest }$.
\end{lemma}
This, combined with the Sandwich Theorem, gives the following inequality:
\begin{theorem}
The number of orientations of $G$ that yield an a-cyclic graph is at most the number of its subgraphs that are forests.
\end{theorem}
These characterizations of $\str(\neg\s)$ and $\sstr(\s)$ also give a natural interpretation of the VC-dimension and the {\it Dual VC-dimension}. The dual VC-dimension of a system is obtained by applying the Duality Tranformation on the text defining the VC-dimension:
\begin{definition}[Dual-VC-dimension]
The \deftext{Dual VC-dimension} (Vapnik Chervonenkis dimension) of a system $\s$, denoted $\dvc(\s)$, is the cardinality of the largest\footnote{As a special case, $\dvc(\s)=-1$ when $S(\s)=\emptyset$} subset that is strongly-shattered by it. 
\end{definition}
The next lemmas easily follows from the characterization of $\sstr(\s)$ and $\str(\neg\s)$. 
\begin{lemma}
Let $\lvert V\rvert=n$ and let $\lvert E\rvert=m$. Then
\begin{enumerate}
\item{$\vc(\neg\s)=k$ where $k$ is the size of a maximum subforest of $G$.}
\item{$\dvc(\s)=m-c$, where $c$ is the size of a smallest cycle in $G$.}
\end{enumerate}
\end{lemma}

\subsection{Path preserving orientations}
The second example will demonstrate a proof for the following equality:
\begin{theorem}\label{thm:second graphs example}
Let $W\subseteq V$ and let $s\in W$. Then, the number of orientations of $G$ in which all vertices in $W$ are reachable from $s$ equals to the number of subgraphs of $G$ in which all vertices of $W$ lie in the same connected component.
\end{theorem}
Let $\s$ denote the system $\sys{\{d\in\{+,-\}^E~:~\forall w\in W:G_d\mbox{ has an }s-w\mbox{ directed path}\}}{\{+,-\}^E}$
The following lemma characterizes $\str(\s)$ and $\sstr(\s)$. Moreover, it shows that $SE(\s)$.
\begin{lemma}\label{lem:digraphs second example SE}
Let $X\subseteq E$. Then, the following three statements are equivalent:
	\begin{enumerate}
	\item{$\s$ strongly shatters $X$}
	\item{$\s$ shatters $X$}
	\item{All vertices of $W$ lie in the same connected componenet of $G_{E-X}$}
	\end{enumerate}
\end{lemma}
\begin{proof}
$1\implies 2$ is trivial.

Consider $2\implies 3$. To establish this implication, we prove its contra-position, namely we assume that not all vertices of $W$ lie in the same connected component of $E-X$ and we show that there exist $d\in\{+,-\}^X$ such that no orientation in $\s$ agrees with $d$ on $X$. Let $w\in W$ such that $w$ and $s$ lie in different components of $E-X$ and let $S$ denote the connected componenet of $S$ . Let $d$ denote an orientation of $X$ in which $w$ is not reachable from $s$ (to obtain such an orientation, orient all edges connecting $S$ and $V-S$ towards $S$). Clearly, every orientation of $E$, that agrees with $d$ on $X$ has no directed path from $S$ to $V-S$ and in particular no directed path from $s\in S$ to $w\in V-S$. Thus, every such orientation is not in $\s$. 

Consider $3\implies 1$. Assume that all vertices of $W$ lie in the same connected componenet of $G_{E-X}$. Let $d$ be an orientation of $E-X$ such that $s$ is a root of its connected component. Note that in particular this means that every $w\in W$ is reachable from $s$. Clearly, in every extension of $d$ to an orientation of $E$, every $w\in W$ is reachable from $s$. 
\end{proof}
Similarly to the first example, the Sandwich Theorem establishes Theorem \ref{thm:second graphs example}.

The following lemma is an easy corollary of Theorem \ref{lem:digraphs second example SE}.
\begin{lemma}
$\vc(\s)=m-t$ where $t$ is the size of a Steiner tree for $W$.
\end{lemma}
Consider the case in which $W=\{s,t\}$. In this case, $\vc(\s)=m-d$ where $d$ is the size of a shortest $s-t$ path in $G$. Moreover, $\neg\s$ is the set of all orientations that do not contain an $s-t$ directed path, and by Lemma \ref{obs:dualStrSstr} we obtain:
\begin{lemma}
The number of orientations that do not contain a directed path from $s$ to $t$ equals to the number of subsets of $E$ that separates $s$ and $t$. Moreover, $\vc(\neg\s)$ is $m-c$, where $c$ is the size of a minimum $s-t$ cut.
\end{lemma}

\subsubsection{An interesting class of $SE$ systems}
In the preceding chapters we showed that $SE$ systems are preserved by many operations (e.g complementing, restrictions, Boolean operators,...). A natural question is to ask whether the class of $SE$ systems is closed under union and intersection. It is not hard to design to $SE$ systems, $\s'$ and $\s''$ that demonstrates that it is not the case. However, there are some subclasses of $SE$ systems that are closed under these operations. One trivial example is the class of all edge-sorted systems. In this subsection we introduce another such class.

Pick some vertex $s\in V$ and let $v\in V-\{x\}$. Let $\s_{s,v}$ be the system of all orientations which yield a directed path from $s$ to $v$. Consider the class
$$\{\s_{s,v}~:~v\in V-\{s\}\}$$
By Lemma \ref{lem:digraphs second example SE} it follows that every system in this class is $SE$. Moreover, by the same lemma we get that any intersection of systems in this class is $SE$. It is not hard to prove, using the same method, that every union of elements in this class is $SE$. Therefore, this forms a non-trivial example for a class of $SE$ systems that is closed under unions and intersections.
\bibliographystyle{plain}
\bibliography{Shay}
\end{document}